\documentclass[12pt,a4paper]{article}
\usepackage{amsmath,amssymb,amsthm,amsfonts}
\usepackage{graphicx}
\usepackage{hyperref}
\usepackage{mathtools}
\usepackage{enumitem}
\usepackage{xcolor}
\hypersetup{
  colorlinks=true,
  linkcolor=blue!50!black,
  citecolor=blue!50!black,
  urlcolor=blue!50!black,
  pdftitle={Discreteness of the Steklov Spectrum for Exterior Non-Compact Free-Boundary Minimal Surfaces with Regular Ends of Finite Total Curvature},
  pdfauthor={Igor Kozyrev},
  pdfsubject={Spectral geometry; Steklov eigenvalue problem; free-boundary minimal surfaces},
  pdfkeywords={Steklov spectrum, exterior free-boundary minimal surface, conformal compactification, Krein formula, Lagrangian normalization, compact resolvent}
}
\newtheorem{theorem}{Theorem}[section]
\newtheorem{lemma}[theorem]{Lemma}
\newtheorem{proposition}[theorem]{Proposition}
\newtheorem{corollary}[theorem]{Corollary}
\newtheorem{definition}[theorem]{Definition}
\newtheorem{remark}[theorem]{Remark}
\newtheorem{example}[theorem]{Example}

\title{Discreteness of the Steklov Spectrum for Exterior Non-Compact Free-Boundary Minimal Surfaces with Regular Ends of Finite Total Curvature}
\author{Igor Kozyrev\\[4pt]
{\normalsize Independent University of Moscow}\\
{\normalsize Bolshoy Vlasyevskiy per.~11, Moscow 119002, Russia}\\
{\normalsize\texttt{ikozyrev@gmail.com}}}
\date{}

\makeatletter
\renewcommand{\maketitle}{\par
  \begingroup
    \renewcommand\thefootnote{\@fnsymbol\c@footnote}%
    \def\@makefnmark{\rlap{\@textsuperscript{\normalfont\@thefnmark}}}%
    \long\def\@makefntext##1{\parindent 1em\noindent
            \hb@xt@1.8em{%
                \hss\@textsuperscript{\normalfont\@thefnmark}}##1}%
    \newpage
    \global\@topnum\z@
    \@maketitle
    \@thanks
  \endgroup
  \setcounter{footnote}{0}%
  \global\let\thanks\relax
  \global\let\maketitle\relax
  \global\let\@maketitle\relax
  \global\let\@thanks\@empty
  \global\let\@author\@empty
  \global\let\@date\@empty
  \global\let\@title\@empty
  \global\let\title\relax
  \global\let\author\relax
  \global\let\date\relax
  \global\let\and\relax
}
\def\@maketitle{%
  \vspace*{-6em}%
  \begin{center}%
  \let \footnote \thanks
    {\LARGE \@title \par}%
    \vskip 1em%
    {\large \@author}%
  \end{center}%
  \par
  \vskip 1em}
\makeatother

\begin{document}

\maketitle

\begin{abstract}\sloppy
We study the Steklov problem on non-compact exterior free-boundary minimal surfaces. Boundary values need not determine a unique harmonic extension, so the operator also requires a prescription at infinity. For proper surfaces in $\mathbb R^3$ with compact boundary and finitely many regular ends of finite total curvature, we construct a natural class of such prescriptions. Every resulting operator is self-adjoint with compact resolvent; consequently, its spectrum is discrete, bounded below, and tends to $+\infty$. If the coordinate functions have linearly independent boundary traces, the prescription can be chosen so that these traces are eigenfunctions with eigenvalue $-1$.
\end{abstract}

\section{Introduction}

The Steklov eigenvalue problem \cite{Steklov1902} places the spectral parameter in the boundary condition: for a Riemannian manifold $\Omega$ with sufficiently smooth boundary, one seeks $u$ and $\sigma$ with

\begin{align}
\Delta u &= 0 \quad \text{in } \Omega \\
\frac{\partial u}{\partial \nu} &= \sigma u \quad \text{on } \partial\Omega,
\end{align}
where $\Delta$ is the Laplace--Beltrami operator and $\nu$ is the outward unit conormal to the boundary (see Section~\ref{sec:notation}). This problem is equivalent to finding the spectrum of the Steklov operator $\mathcal{S}$ (defined precisely in Section~\ref{sec:steklov_def}), which maps a function $f$ on the boundary to the conormal derivative of its normalized harmonic extension into the domain.

For compact connected manifolds with sufficiently smooth boundary, the Steklov spectrum is discrete and consists of a sequence $0=\sigma_0<\sigma_1\leq\sigma_2\leq\cdots\to\infty$. On a non-compact manifold the harmonic extension need not be unique until behavior at infinity is prescribed; the resulting spectrum depends on that choice and on the geometry at infinity.

A foundational connection between the Steklov eigenvalue problem and minimal surfaces was established by Fraser and Schoen \cite{FraserSchoen2011,FraserSchoen2016}. They proved that metrics maximizing the first non-zero Steklov eigenvalue (normalized by boundary length) on a compact oriented surface of genus zero with $k\geq2$ boundary components, or on the M\"obius band, admit a branched conformal free-boundary minimal immersion into a Euclidean unit ball \emph{by first eigenfunctions}, the maximizing metric being a $\sigma$-homothety of the induced one --- conformal to it, with conformal factor constant on the boundary \cite[Theorem~5.16]{FraserSchoen2016}. The distinction matters: a conformal change $e^{2h}g$ with $h$ vanishing near $\partial\Sigma$ leaves both the Steklov operator and the boundary length untouched while changing the interior factor arbitrarily, so ``induced after one global rescaling'' would be too strong. Such an immersion is minimal and meets the boundary sphere orthogonally, and its coordinate functions are Steklov eigenfunctions with eigenvalue~$1$ \cite[Lemma~2.2]{FraserSchoen2011}. For a maximizer these are first eigenfunctions by construction, so that $\sigma_1 = 1$; for a free-boundary minimal surface not known to be a maximizer this is instead a conjecture of Fraser and Li \cite{FraserLi2014}, stated there for compact \emph{properly embedded} minimal hypersurfaces of the ball and verified for the model examples; for embedded surfaces of genus zero in $\mathbb B^3$ it has been proved by Chodosh and Gianocca \cite{ChodoshGianocca2026}, who deduce that the critical catenoid is the unique embedded free-boundary minimal annulus. The embeddedness hypothesis cannot be dropped: on the $m$-fold immersed cover of the critical catenoid, $m\geq2$, parametrized by $t\in[-T,T]$ and $\theta\in\mathbb R/2\pi m\mathbb Z$ with $T\tanh T=1$, the function $\cosh(t/m)\cos(\theta/m)$ is harmonic and Steklov with eigenvalue $\tanh(T/m)/(m\tanh T)\in(0,1)$, while the coordinates still have eigenvalue $1$; so the coordinates are not first eigenfunctions there --- which is why the rigidity theorems of \cite{FraserSchoen2016} take ``immersed by first eigenfunctions'' as a hypothesis.

In this paper, we consider a complementary setting: \emph{exterior} free-boundary minimal surfaces $\Sigma\subset\mathbb R^3\setminus B^3$ that lie outside the open unit ball, have $\partial\Sigma\subset S^2$, and meet the sphere orthogonally from the outside; we abbreviate ``free-boundary minimal surface'' to FBMS. Mazet and Mendes proved that every properly embedded exterior FBMS with one regular end is a rotational surface $C_\alpha$ (see Section~\ref{subsec:mazet_mendes}), and obtained related rigidity for stable surfaces with parallel regular ends \cite{MazetMendes2022}. Our focus is the spectral-theoretic side. The outward conormal along $\partial\Sigma$ points toward the ball center, so $\nu=-\mathbf x$ and $\partial_\nu x_i=-x_i$. Because harmonic extension is non-unique, this boundary identity makes the coordinate traces eigenfunctions with eigenvalue $-1$ only for a normalization that retains the coordinate functions. Such a normalization exists when their boundary traces are linearly independent. Other normalizations can have no negative spectrum.

In a related direction, Medvedev and Morozov studied the Fraser--Sargent exterior surfaces $EFS_{k,l}\subset\mathbb R^4\setminus B^4$, which are codimension-two examples outside the scope of the Mazet--Mendes rigidity, and proved that they are stable \cite[Theorem~1.1]{MedvedevMorozov2022}. In~\cite[§1.2, item~6]{MedvedevMorozov2022v2} they note that non-compactness leaves the discreteness of the Steklov spectrum, and hence a spectral index, unresolved. The present paper constructs self-adjoint realizations with discrete spectrum once behavior at infinity is fixed. Each realization therefore has a finite counting function below every fixed threshold; choosing a geometrically meaningful threshold for an exterior spectral index is a separate, normalization-dependent question. The core analysis is carried out for hypersurfaces in $\mathbb R^3\setminus B^3$, while Section~\ref{sec:higher_codim} proves the intrinsic discreteness theorem for complete Riemannian surfaces with intrinsic regular ends, at any choice of orders. It then specializes the coordinate-eigenfunction statement to exterior FBMS in arbitrary codimension. The intrinsic theorem applies, in particular, to every $EFS_{k,l}$; Section~\ref{subsec:FS_example} states separately which ambient coordinates are admissible at a given order.

More precisely, in $\mathbb R^3$ we assume finitely many regular ends of finite total curvature. Huber--Osserman theory gives the punctured conformal structure and meromorphic data \cite{Huber1957,Osserman1964}; for the properly embedded annular multiplicity-one ends considered here, Schoen's theorem supplies the graph expansion and the catenoidal or planar alternatives \cite[Proposition~1]{Schoen1983}. That asymptotic analysis is the geometric input of this paper, and it is all that the ends contribute: once each end is known to be a conformal punctured disk on which the coordinate functions have at most a simple principal part, the spectral theory is elementary. In dimension two harmonicity and conormal flux are conformally invariant and bounded harmonic functions extend across punctures, so the punctures can simply be filled. The Steklov operator of the resulting \emph{compact} surface is classical, and every admissible realization differs from it by a symmetric operator of finite rank.

For the Laplace spectrum on non-compact minimal surfaces, discreteness phenomena were studied by Bessa--Jorge--Montenegro \cite{BessaJorgeMontenegro2010} and Bessa--Jorge--Mari \cite{BessaJorgeMari2015}; in the Steklov setting the problem is more intricate because the eigenvalue parameter occurs at the boundary.

Recently, Bundrock, Girouard, Grebenkov, Levitin, and Polterovich systematically studied the exterior Steklov problem on flat Euclidean domains $\mathbb R^n\setminus\bar\Omega$ \cite{BGGLP2026}. Their far-field condition $u(x)=O(|x|^{2-n})$ --- boundedness for $n=2$ and decay for $n\geq3$ --- makes harmonic extension unique but excludes the linearly growing coordinate functions. Our one-end coordinate normalization retains those functions. These choices are different self-adjoint realizations and already produce different spectra on the exterior of the equatorial disk (Remark~\ref{rem:comparison_BGGLP}).

Our main results are the following.

\textbf{Theorem~\ref{thm:krein}} (Section~\ref{sec:steklov_weighted}).
\emph{Every symmetric Steklov realization obtained by a linear selection of admissible harmonic extensions is $S_M=S_0+\Psi M\Psi^\top$ for a unique symmetric matrix $M$, where $S_0$ is the Dirichlet-to-Neumann map of the conformal compactification and $\Psi$ is built from the conormal derivatives along $\partial\Sigma$ of the harmonic functions with zero boundary values and prescribed principal part at one puncture. The Lagrangian complements of the Dirichlet-kernel data correspond bijectively to symmetric matrices.}

\textbf{Theorem~\ref{thm:main_body}} (Section~\ref{sec:proof_discrete}).
\emph{For a properly embedded exterior FBMS with compact boundary and one regular end, every Lagrangian normalization gives a self-adjoint Steklov operator with compact resolvent and discrete spectrum bounded below and tending to $+\infty$. Unless the surface lies in a plane through the origin, exactly one normalization makes the coordinate traces eigenfunctions with eigenvalue $-1$, and for it the negative spectrum is exactly $\{-1\}$ with multiplicity three.}

More generally, the discreteness part holds for any finite number of ends, and the counting functions of all realizations differ from that of $S_0$ by at most the rank of $M$; with several ends the coordinate-compatible normalization is no longer unique (Section~\ref{sec:multiple_ends}).

\textbf{Theorem~\ref{thm:multiple_ends}} (Section~\ref{sec:multiple_ends}).
\emph{For an exterior FBMS with compact boundary and finitely many regular ends, every Lagrangian normalization gives a self-adjoint Steklov operator with compact resolvent and discrete spectrum; a coordinate-compatible normalization always exists, and for it every non-zero coordinate trace is an eigenfunction with eigenvalue $-1$; all three traces are non-zero unless the surface lies in a plane through the origin.}

Compactness of the resolvent rests on the compactness of $\partial\Sigma$ and on the ends being finitely many conformal punctured disks, through the corresponding property of the Dirichlet-to-Neumann map of the compactified surface. Within the class considered, the spectral conclusion is therefore independent of the number and type of ends; what the ends control is the dimension of the family of realizations.

The present paper is a first step toward an exterior Fraser--Schoen theory. After fixing a coordinate-compatible normalization, let $\sigma_{-1}$ be the largest negative eigenvalue of the resulting Steklov operator. A candidate functional is $|\sigma_{-1}|\operatorname{Length}(\partial\Sigma)$. With several ends no canonical coordinate-compatible normalization is selected, so the metric alone does not determine this quantity. For a fixed realization, compact resolvent makes $\sigma_{-1}$ an isolated eigenvalue and permits variational methods; developing the extremal theory is left to future work.

The paper is organized as follows. Section~\ref{sec:strategy} outlines the strategy. Section~\ref{sec:notation} fixes the setting and the regular ends, and Appendix~\ref{app:huber} their conformal structure; this is the geometric part. Section~\ref{sec:compactification} fills the punctures and describes the admissible harmonic functions, Section~\ref{sec:steklov_weighted} proves the Kre\u\i n formula and the Lagrangian dictionary, and Section~\ref{sec:proof_discrete} proves the spectral theorem and Theorem~\ref{thm:main_body}. Section~\ref{sec:applications} discusses examples. Section~\ref{sec:multiple_ends} records the multi-end case, which needs no new analysis; Section~\ref{sec:higher_codim} treats arbitrary codimension and the Fraser--Sargent surfaces. The appendix contains the deferred geometric details.

\section{Strategy of the Proof}\label{sec:strategy}
We treat the Steklov operator as an \emph{unbounded} operator with \emph{compact resolvent}, not as a compact operator. The argument has five steps, of which only the first is geometric.
\begin{enumerate}
  \item \textbf{The ends are conformal punctured disks.} Finite total curvature and the regular-end expansion give, on each end, a conformal coordinate $\zeta$ in which the end is $\{0<|\zeta|<\varepsilon\}$ and the coordinate functions of an exterior FBMS have at most the principal parts $-\log|\zeta|$, $\operatorname{Re}\zeta^{-1}$, $\operatorname{Im}\zeta^{-1}$ (Section~\ref{sec:notation}, Lemma~\ref{lem:conformal_end}).
  \item \textbf{Fill the punctures.} Harmonicity is conformally invariant in dimension two and bounded harmonic functions extend across a puncture, so the surface compactifies to $(\overline\Sigma,\hat g)$ with $\hat g=g$ near the boundary. The admissible harmonic functions are those with a prescribed principal part at each puncture; they form a space $\mathcal H$ with $\dim K=N$ Dirichlet kernel, spanned by explicit correctors (Section~\ref{sec:compactification}).
  \item \textbf{Kre\u\i n formula.} Every linear selection of admissible extensions gives $S=S_0+\Psi G$ with $S_0$ the Dirichlet-to-Neumann map of $(\overline\Sigma,\hat g)$, and $S$ is symmetric exactly when $G=M\Psi^\top$ with $M$ symmetric (Theorem~\ref{thm:krein}). The Lagrangian complements of the kernel data are in bijection with such $M$ (Proposition~\ref{prop:boundary_normalization}).
  \item \textbf{Spectrum.} $S_0$ is a classical elliptic pseudodifferential operator of order one on the compact boundary, self-adjoint with compact resolvent; $\Psi M\Psi^\top$ is bounded, symmetric and smoothing. Hence every $S_M$ inherits self-adjointness, the principal symbol, compactness of the resolvent and discreteness, and Glazman's lemma bounds the counting functions against each other (Theorem~\ref{thm:spectral}).
  \item \textbf{Coordinates.} The free-boundary identity $\partial_\nu x_i=-x_i$ makes the coordinate data isotropic, with zero intersection with the kernel data, so a coordinate-compatible Lagrangian always exists; it is unique exactly when the coordinate data fill the available dimension, and then the negative spectrum is exactly $\{-1\}$ (Proposition~\ref{prop:coordinate_eigenfunctions}).
\end{enumerate}

Neither weighted Sobolev spaces, which \cite[Remark~3.7]{BGGLP2026} mention as a possible route to uniqueness of harmonic extensions, nor Fredholm theory on them is used. The treatment by conformal maps to punctured disks, as in \cite{BGGLP2026}, which removes the need for the weighted machinery, was suggested by I.~Polterovich; see the Acknowledgements.

\section{Setup and Preliminaries}\label{sec:notation}

We fix the sign conventions, exterior geometry, and regular-end coordinates used in the analytic construction.

\paragraph{Asymptotic notation.}
For non-negative quantities $F$ and $G$ depending on a variable tending to infinity, we write $F\lesssim G$ if $F\leq CG$ for some constant $C>0$ independent of the asymptotic variable, and $F\asymp G$ if both $F\lesssim G$ and $G\lesssim F$. We write $F\sim G$ if $G$ is non-zero for all sufficiently large values of the variable and $F/G\to1$. When angular variables are present, these estimates and limits are uniform in those variables unless stated otherwise. Constants may depend on the fixed surface, end, or differentiation order, but not on the variable tending to infinity; in estimates involving a varying function, independence of that function is stated when needed. The symbol over an arrow, as in $K\xrightarrow{\sim}\mathbb R^3$, denotes an isomorphism and is unrelated to asymptotic equivalence.

\subsection{Geometric setting}

We work in Euclidean space $\mathbb{R}^3$ with the standard inner product $\langle \cdot, \cdot \rangle$. We denote by $B^3 = \{\mathbf{x} \in \mathbb{R}^3 : |\mathbf{x}| < 1\}$ the open unit ball and by $S^2 = \partial B^3$ the unit sphere. The position vector of a point $p \in \mathbb{R}^3$ is denoted $\mathbf{x}(p) = p$, and we write $\mathbf{x} = (x_1, x_2, x_3)$ for the coordinate functions. We use $e_1, e_2, e_3$ for the standard basis vectors.

\subsection{Surfaces and the Laplace--Beltrami operator}

Throughout, $(\Sigma, g)$ denotes a smooth oriented Riemannian surface (possibly with boundary and non-compact), where $g$ is the induced metric from $\mathbb{R}^3$. We use the \textbf{geometer's sign convention} for the Laplace--Beltrami operator:
\begin{equation}\label{eq:laplace_beltrami}
\Delta_g f = \frac{1}{\sqrt{\det g}} \sum_{i,j} \frac{\partial}{\partial x^i}\left(\sqrt{\det g}\, g^{ij} \frac{\partial f}{\partial x^j}\right) = \mathrm{div}_g(\nabla_g f),
\end{equation}
so that the eigenvalues of $-\Delta_g$ on a closed manifold are non-negative, and $\Delta_g f = 0$ means $f$ is \emph{harmonic}.

The unit normal is denoted $\mathbf n$. For an immersion $X$, minimality is equivalent to $\Delta_gX=0$.

\subsection{Regular ends and polar coordinates}\label{subsec:planar_ends}\label{subsec:polar_coords}

Following Mazet and Mendes \cite[Section~2.2]{MazetMendes2022}, we use Schoen's notion of a \emph{regular end}~\cite{Schoen1983}. In this subsection $\Sigma \subset \mathbb{R}^3$ is a hypersurface and we write $P=(X,z)\in\mathbb{R}^2\times\mathbb{R}$. An end $E$ is regular if, after a Euclidean isometry, it is the graph of a bounded-gradient function on $\{|X|\geq R\}$ with
\begin{equation}\label{eq:regular_end_n2}
f(X) = A \log |X| + B + \langle C, X \rangle |X|^{-2} + O(|X|^{-2}),
\end{equation}
where $A,B\in\mathbb{R}$ and $C\in\mathbb{R}^2$. Writing
\[
R_f(X):=f(X)-A\log|X|-B-\langle C,X\rangle|X|^{-2},
\]
the differentiated form of the expansion means that, for every multi-index $\alpha$,
\[
|D^\alpha R_f(X)|\leq C_\alpha |X|^{-2-|\alpha|}
\qquad (|X|\ \text{large}).
\]
These derivative bounds are part of the regular-end hypothesis used below. They give $|\mathrm{II}|=O(r^{-2})$ and the coordinate expansions of Lemma~\ref{lem:conformal_end}. Finite total curvature gives the punctured conformal type through Huber--Osserman, but does not by itself give the graph expansion. For a properly embedded annular multiplicity-one end of finite total curvature, the expansion is supplied by \cite[Proposition~1]{Schoen1983}; the finite-total-curvature assumption is part of that proposition. This is the regularity hypothesis used in the $\mathbb{R}^3$ results below. Higher codimension uses the separate intrinsic Definition~\ref{def:intrinsic_end}.

Since $dA \asymp r\,dr\,d\phi$ and minimality gives $|K|=\tfrac12|\mathrm{II}|^2=O(r^{-4})$, every regular end has finite total curvature. Hence a surface with compact boundary and finitely many regular ends satisfies the Huber--Osserman hypothesis (Appendix~\ref{app:huber}, Lemma~\ref{lem:conformal_end}).

On such an end we use \textbf{polar coordinates} $(r, \phi)$, $X = (r\cos\phi, r\sin\phi)$, $r = |X| \geq R$: the map $(r, \phi) \mapsto (r\cos\phi, r\sin\phi, f(r,\phi))$ is a coordinate chart on the end; the conformal coordinate of Lemma~\ref{lem:conformal_end} is related to it by $\log r=-\log|\zeta|+O(1)$.

A \emph{catenoidal end} is a regular end with $A\neq0$; in polar coordinates,

\begin{equation}\label{eq:catenoidal_polar}
z = a \log r + b + O(r^{-1}) \quad \text{as } r \to \infty,
\end{equation}
where $a\neq0$ is the logarithmic growth rate.

A \emph{planar end} is a regular end with $A=0$. The defining expansion then gives
\begin{equation}
z = b + c(\phi)\,r^{-1} + O(r^{-2}) \quad \text{as } r \to \infty,
\end{equation}
where $b$ is a constant and $c(\phi) = c_1\cos\phi + c_2\sin\phi$ is the coefficient coming from the term $\langle C, X\rangle |X|^{-2}$ of~\eqref{eq:regular_end_n2}. The coordinate function $z$ is bounded on the end (in contrast to the logarithmic growth on a catenoidal end).

The exterior of the equatorial disk $C_0 = \{z = 0\} \setminus B^3$ (the plane with the \emph{open} unit disk removed, i.e.\ $\{|X| \geq 1,\ z = 0\}$ --- a non-compact surface with one planar end) is the unique properly embedded \emph{one-ended} exterior FBMS with a planar regular end \cite{MazetMendes2022} (the classification there concerns surfaces with one regular end); its spectrum is computed in Example~\ref{ex:flat_disk}.

\subsection{Exterior free-boundary minimal surfaces}

A free-boundary minimal surface in the unit ball $B^3 \subset \mathbb{R}^3$ is a minimal surface $\Sigma$ with $\partial\Sigma \subset S^2$ that meets the unit sphere orthogonally. In this paper, we consider the \emph{exterior} variant:

The \textbf{outward unit conormal} $\nu$ at $p \in \partial\Sigma$ is the unique unit vector in $T_p\Sigma$ perpendicular to $T_p(\partial\Sigma)$ and pointing away from $\Sigma$; it is \emph{not} the surface normal $\mathbf{n}$.

\begin{definition}\label{def:exterior_FBMS}
An \emph{exterior free-boundary minimal surface} is a smooth surface $\Sigma$ together with a proper minimal immersion $F\colon\Sigma\to\mathbb{R}^3\setminus B^3$ such that:
\begin{itemize}
    \item $F^{-1}(S^2)=\partial\Sigma$ (the immersion is neat along the sphere);
    \item $F$ meets $S^2$ orthogonally from the outside along $\partial\Sigma$;
    \item every connected component of $\Sigma$ has non-empty boundary.
\end{itemize}
We equip $\Sigma$ with $g=F^*g_{\mathrm{eucl}}$, suppress $F$ from the notation, and write $\mathbf{x}=F$. When $F$ is embedded, we identify $\Sigma$ with its image. In Section~\ref{sec:higher_codim} the same definition is used verbatim with $(\mathbb{R}^3, B^3, S^2)$ replaced by $(\mathbb{R}^n, B^n, S^{n-1})$, orthogonality in higher codimension meaning that the position vector is tangent to $\Sigma$ along the boundary, $\mathbf{x}(p) \in T_p\Sigma$ for $p \in \partial\Sigma$ (equivalently, $\mathbf{x}$ is orthogonal to the normal space of $\Sigma$), which for a hypersurface is exactly $\langle \mathbf{n}, \mathbf{x}\rangle = 0$; the conormal identity $\nu = -\mathbf{x}$ and hence $\partial_\nu x_k = -x_k$ along $\partial\Sigma$ are unchanged, since their derivation uses only that $|\mathbf{x}| = 1$ and $\mathbf{x} \in T_p\Sigma$ on the boundary, together with $\Sigma$ lying outside the ball --- the last fixing the sign in $\nu = \pm\mathbf{x}$.
\end{definition}

Orthogonality means $\langle \mathbf{n}, \mathbf{x} \rangle = 0$ along $\partial\Sigma$; since $\Sigma$ lies outside $B^3$, the outward conormal is $\nu = -\mathbf{x}$ (it points toward the ball center), in contrast to $\nu = +\mathbf{x}$ in the interior case --- the source of the coordinate eigenvalue $\sigma=-1$ in a coordinate-compatible realization (Section~\ref{subsec:coord_eigen}).

The third condition excludes components invisible to $L^2(\partial\Sigma)$, such as a plane disjoint from $B^3$; the boundary operator would otherwise ignore their independent interior harmonic degrees of freedom.

\section{Conformal Compactification and Admissible Harmonic Functions}\label{sec:weighted_sobolev}\label{sec:compactification}

\emph{Standing hypotheses.} From here until Section~\ref{sec:multiple_ends}, $\Sigma$ is a smooth Riemannian surface with smooth compact boundary, every connected component of $\Sigma$ meets $\partial\Sigma$, and outside a compact set $\Sigma_0\supset\partial\Sigma$ the surface is a disjoint union of $L\geq1$ ends $E_1,\dots,E_L$, each carrying a conformal diffeomorphism
\[
\zeta_\ell\colon E_\ell \xrightarrow{\ \sim\ } D^*_\varepsilon=\{0<|\zeta|<\varepsilon\},
\qquad (\zeta_\ell^{-1})^*g=\mu_\ell\,|d\zeta|^2,
\]
with $\mu_\ell\in C^\infty(D^*_\varepsilon)$ positive, such that the punctures lie at infinity: for every $0<\delta<\varepsilon$ the set $\Sigma_0\cup\bigcup_\ell\zeta_\ell^{-1}(\{\delta\leq|\zeta|<\varepsilon\})$ is compact, equivalently $\zeta_\ell(p)\to0$ as $p\in E_\ell$ leaves every compact subset of $\Sigma$.
No completeness, curvature bound or metric asymptotics are assumed. By Lemma~\ref{lem:conformal_end} every exterior FBMS with compact boundary and finitely many regular ends satisfies these hypotheses, so the analysis below applies to the surfaces of Section~\ref{sec:notation}; this is the only place where the geometry of Section~\ref{sec:notation} is used.

\subsection{Admissible growth at a puncture}\label{subsec:weight_choice}

Fix orders $N_\ell\geq1$ and put, on the end $E_\ell$,
\begin{equation}\label{eq:principal_parts}
s_{\ell,0}=-\log|\zeta_\ell|,\qquad
s_{\ell,2j-1}=\operatorname{Re}\zeta_\ell^{-j},\qquad
s_{\ell,2j}=\operatorname{Im}\zeta_\ell^{-j}
\qquad (1\leq j\leq N_\ell).
\end{equation}
Let $I=\{(\ell,m):1\leq\ell\leq L,\ 0\leq m\leq 2N_\ell\}$ and $N=|I|=\sum_\ell(2N_\ell+1)$. We write $C^\infty(\Sigma)$ for functions smooth up to $\partial\Sigma$, $\nu$ for the outward unit conormal and $\langle\cdot,\cdot\rangle$ for the $L^2(\partial\Sigma,ds)$ pairing.

\begin{definition}\label{def:weight_function}\label{def:admissible}
$\mathcal H$ is the space of $u\in C^\infty(\Sigma)$ with $\Delta_gu=0$ such that for each $\ell$ there is $c_\ell(u)\in\mathbb R^{2N_\ell+1}$ with $u-\sum_m c_{\ell,m}(u)s_{\ell,m}$ bounded on $\{0<|\zeta_\ell|<\varepsilon/2\}$. Put $c(u)=(c_\ell(u))_\ell\in\mathbb R^N$ and
\[
K=\{u\in\mathcal H:\ u|_{\partial\Sigma}=0\}.
\]
\end{definition}

If $\zeta'=a\zeta+O(\zeta^2)$ is another conformal puncture coordinate then $\zeta'^{-j}$ is a polynomial of degree $j$ in $\zeta^{-1}$ plus a bounded term, and $\log|\zeta'|-\log|\zeta|$ is bounded; for an antiholomorphic coordinate $\zeta'=a\bar\zeta+O(\bar\zeta^2)$ the same holds with $\bar\zeta^{-1}$, whose powers have the same real and imaginary parts up to sign. Hence $\mathcal H$ depends only on the conformal structure and on the orders $(N_\ell)$.

\begin{remark}\label{rem:planar_weights}\label{rem:orders}
For the regular ends of Section~\ref{subsec:planar_ends} the relevant order is $N_\ell=1$: by Lemma~\ref{lem:conformal_end} a rotated frame gives $x_1-ix_2=c\,\zeta^{-1}+O(1)$ and $x_3=a\log r+O(1)$ with $\log r=-\log|\zeta|+O(1)$, so every coordinate function of an exterior FBMS lies in $\mathcal H$ with $N_\ell=1$, and no higher principal part is needed. Where a surface of higher codimension needs a larger order, this is stated explicitly (Section~\ref{subsec:FS_example}). The choice of $(N_\ell)$ is the only input that fixes how much growth is retained at infinity. (In a weighted-Sobolev formulation of the same problem, with norms $\int|u|^2\rho^{2\delta}\,dV$ and $\rho$ the radial function of the end, $N_\ell=1$ corresponds to the weight window $\delta\in(-3,-2)$.)
\end{remark}

We use three standard facts.
\begin{enumerate}[label=(F\arabic*)]
\item\label{F1} In dimension two $\Delta_{\varphi h}=\varphi^{-1}\Delta_h$ for $\varphi>0$. In particular, on $E_\ell$ a $g$-harmonic function is Euclidean harmonic in $\zeta_\ell$, and the $s_{\ell,m}$ are harmonic.
\item\label{F2} A bounded harmonic function $v$ on $D^*_\varepsilon$ extends harmonically to $D_\varepsilon$. (Let $P$ be the Poisson integral of $v$ on $D_{\varepsilon/2}$; for $\eta>0$ the function $\pm(v-P)-\eta\log\frac{\varepsilon}{2|\zeta|}$ is harmonic on $D^*_{\varepsilon/2}$, vanishes on $|\zeta|=\varepsilon/2$ and tends to $-\infty$ at $0$, hence is $\leq0$; let $\eta\to0$.)
\item\label{F3} Let $(\Omega,h)$ be a compact Riemannian surface with smooth boundary, each component of which meets $\partial\Omega$. Its Dirichlet-to-Neumann map $\Lambda f=\partial_\nu u_f$ is symmetric and non-negative on $C^\infty(\partial\Omega)$, with $\langle\Lambda f,f\rangle=\int_{\Omega}|\nabla u_f|^2$; it is a classical elliptic pseudodifferential operator of order one with principal symbol $|\xi|$, and its closure is self-adjoint on $L^2(\partial\Omega)$ with domain $H^1(\partial\Omega)$ and compact resolvent \cite{LeeUhlmann1989,GirouardPolterovich2017}.
\end{enumerate}

\subsection{Filling the punctures}

\begin{lemma}[Compactification]\label{lem:compactification}\label{thm:trace}
Let $\overline\Sigma=\Sigma\sqcup\{p_1,\dots,p_L\}$ carry the smooth structure making each $\zeta_\ell$ (with $\zeta_\ell(p_\ell)=0$) a chart. Then $\overline\Sigma$ is a compact surface with boundary $\partial\Sigma$, each of whose components meets $\partial\Sigma$, and it carries a smooth metric $\hat g$ with $\hat g|_\Sigma=\varphi g$ for some positive $\varphi\in C^\infty(\Sigma)$ equal to $1$ near $\partial\Sigma$. A function on an open subset of $\Sigma$ is $g$-harmonic if and only if it is $\hat g$-harmonic.
\end{lemma}

\begin{proof}
Truncate the ends at half the radius and set
\[
C':=\Sigma\setminus\bigcup_\ell\zeta_\ell^{-1}\bigl(\{0<|\zeta|<\varepsilon/2\}\bigr),
\]
a compact subsurface by the standing hypothesis on the punctures: outside the charts $\Sigma$ is its compact core $\Sigma_0$, and inside each chart
$C'$ is the closed annulus $\varepsilon/2\leq|\zeta_\ell|<\varepsilon$ glued to it, the point
$\zeta_\ell=0$ being the end's point at infinity. Then $\overline\Sigma$ is the union of $C'$
with the closed filled disks $\zeta_\ell^{-1}(\{0<|\zeta|\leq\varepsilon/2\})\cup\{p_\ell\}$,
hence compact; a disk minus a point is connected, so the components of $\overline\Sigma$ are
those of $\Sigma$ with the punctures filled. Take $\chi\in C^\infty_c(D_{\varepsilon/2})$ with
$0\leq\chi\leq1$ and $\chi=1$ on $D_{\varepsilon/4}$, set
$\hat g=(\chi+(1-\chi)\mu_\ell)|d\zeta_\ell|^2$ on each chart and $\hat g=g$ elsewhere. Then $\varphi=(\chi+(1-\chi)\mu_\ell)/\mu_\ell$ on $E_\ell$ and $\varphi=1$ elsewhere, and $\operatorname{supp}\chi\circ\zeta_\ell$ is closed and disjoint from the compact set $\partial\Sigma$. The last claim is~\ref{F1}.
\end{proof}

\begin{lemma}[Unique continuation]\label{lem:UCP}
If $u\in C^\infty(\Sigma)$, $\Delta_gu=0$ and $u=\partial_\nu u=0$ on $\partial\Sigma$, then $u\equiv0$.
\end{lemma}

\begin{proof}
Attach an exterior collar to obtain $\Sigma'\supset\Sigma$ with a smooth extension of $g$, and let $\tilde u=u$ on $\Sigma$, $\tilde u=0$ on $\Sigma'\setminus\Sigma$. Then $\tilde u$ is locally Lipschitz, since $u$ is smooth up to $\partial\Sigma$ and vanishes there, and, for $\phi\in C^\infty_c(\Sigma'^\circ)$, $\int_{\Sigma'}\langle\nabla\tilde u,\nabla\phi\rangle=-\int_\Sigma\phi\,\Delta u+\int_{\partial\Sigma}\phi\,\partial_\nu u=0$. So $\tilde u$ is harmonic, vanishes on an open set, and by unique continuation \cite{Aronszajn1957} vanishes on every component of $\Sigma'$ meeting $\Sigma'\setminus\Sigma$, that is, on all of $\Sigma$.
\end{proof}

\subsection{Structure of the admissible space}

\begin{lemma}[Structure of $\mathcal H$]\label{lem:harmonic}\label{lem:structure}
\begin{enumerate}[label=(\alph*)]
\item The coefficients $c(u)$ are unique, $\mathcal H$ is a vector space, $c$ is linear, and the regular part $u-\sum_m c_{\ell,m}(u)s_{\ell,m}$ extends harmonically across $p_\ell$.
\item For $f\in C^\infty(\partial\Sigma)$ let $u^0_f$ be the $\hat g$-harmonic function on $\overline\Sigma$ with trace $f$. Then $u^0_f|_\Sigma\in\mathcal H$, $c(u^0_f)=0$, and $u^0_f$ is the unique bounded $g$-harmonic function in $C^\infty(\Sigma)$ with trace $f$.
\item For each $(\ell,m)\in I$ there is a unique $w_{\ell,m}\in K$ with $c(w_{\ell,m})=e_{\ell,m}$.
\item Every $u\in\mathcal H$ satisfies $u=u^0_{u|_{\partial\Sigma}}+\sum_I c_{\ell,m}(u)\,w_{\ell,m}$. In particular $K=\operatorname{span}\{w_{\ell,m}\}$ and $\dim K=N$.
\item The functions $\psi_{\ell,m}:=\partial_\nu w_{\ell,m}|_{\partial\Sigma}\in C^\infty(\partial\Sigma)$ are linearly independent.
\end{enumerate}
\end{lemma}

\begin{proof}
(a) If $b=\sum_m c_ms_{\ell,m}$ is bounded near $0$, its Fourier coefficients on $|\zeta|=\rho$ in the modes $\cos j\theta,\sin j\theta$ ($j\geq1$) are $\pm c_{2j-1}\rho^{-j}$, $\pm c_{2j}\rho^{-j}$, and in the mode $1$ it is $-c_0\log\rho$; boundedness forces $c=0$. Linearity follows. The regular part is bounded and harmonic by~\ref{F1}, hence extends by~\ref{F2}.

(b) $u^0_f\in C^\infty(\overline\Sigma)$ by elliptic regularity, is $g$-harmonic on $\Sigma$ by Lemma~\ref{lem:compactification}, and is bounded near each $p_\ell$. If $h$ is bounded, $g$-harmonic and has zero trace, then by~\ref{F2} it extends to a $\hat g$-harmonic function on $\overline\Sigma$ vanishing on $\partial\Sigma$, so $h=0$ by the maximum principle, every component of $\overline\Sigma$ meeting $\partial\Sigma$.

(c) Let $\chi_\ell$ be a cutoff equal to $1$ near $p_\ell$ and supported in $\zeta_\ell^{-1}(D_{\varepsilon/2})\cup\{p_\ell\}$. Then $F=\Delta_{\hat g}(\chi_\ell s_{\ell,m})$ vanishes where $\chi_\ell\in\{0,1\}$, hence extends by zero to $C^\infty_c(\overline\Sigma^\circ)$. Let $v$ solve $\Delta_{\hat g}v=-F$ with $v|_{\partial\Sigma}=0$ and put $w_{\ell,m}=\chi_\ell s_{\ell,m}+v$. Uniqueness follows from (b).

(d) $h=u-u^0_{u|_{\partial\Sigma}}-\sum c_{\ell,m}(u)w_{\ell,m}\in\mathcal H$ has $c(h)=0$, hence is bounded, and has zero trace; so $h=0$ by (b).

(e) If $\sum a_{\ell,m}\psi_{\ell,m}=0$ then $k=\sum a_{\ell,m}w_{\ell,m}$ has vanishing Cauchy data, so $k=0$ by Lemma~\ref{lem:UCP} and $a=c(k)=0$.
\end{proof}

\begin{remark}\label{rem:decaying_not_kernel}
Decay and zero boundary data force a kernel element to vanish: by (a) a bounded element of $K$ extends across the punctures and vanishes by the maximum principle. Thus only the principal parts~\eqref{eq:principal_parts} can occur as leading kernel terms. On $C_0$ the basis with $N_1=1$ is $\log r$ and $(r-r^{-1})\{\cos,\sin\}\theta$ (Appendix~\ref{app:kernel_example}).
\end{remark}

\begin{remark}\label{rem:harmonic_planar}
Nothing above distinguishes catenoidal from planar ends: both are conformal punctured disks, and only the conformal structure enters. The distinction reappears in Section~\ref{subsec:coord_eigen}, where the growth rate $a$ decides whether $x_3$ contributes a principal part.
\end{remark}

\paragraph{Notation.} Let $\Psi\colon\mathbb R^N\to C^\infty(\partial\Sigma)$, $\Psi a=\sum a_{\ell,m}\psi_{\ell,m}$, with transpose $\Psi^\top f=(\langle f,\psi_{\ell,m}\rangle)_I$. The Gram matrix $\Gamma=\Psi^\top\Psi$ is invertible by Lemma~\ref{lem:structure}(e); $\Psi$ is injective and $\Psi^\top$ is onto $\mathbb R^N$. Define
\begin{equation}\label{eq:S0_def}
S_0f=\partial_\nu u^0_f|_{\partial\Sigma},\qquad f\in C^\infty(\partial\Sigma).
\end{equation}
Since $\hat g=g$ near $\partial\Sigma$, the conormals and line elements agree, so $S_0$ is exactly the Dirichlet-to-Neumann map of the compact surface $(\overline\Sigma,\hat g)$ and~\ref{F3} applies to it. By Lemma~\ref{lem:structure}(b), $S_0$ is the bounded realization. It is also the finite-energy one. The Dirichlet energy is conformally invariant, and distinct Fourier modes are orthogonal for it on an annulus $\{\rho_0<|\zeta|<\varepsilon\}$. There the energy of $(\alpha\rho^{-j}+\beta\rho^{j})\cos j\phi$, $j\geq1$, is $\pi j\bigl(\alpha^2(\rho_0^{-2j}-\varepsilon^{-2j})+\beta^2(\varepsilon^{2j}-\rho_0^{2j})\bigr)$, the cross terms cancelling (likewise for $\sin j\phi$), and that of $\alpha_0\log\rho$ is $2\pi\alpha_0^2\log(\varepsilon/\rho_0)$. So an element $u$ of $\mathcal H$ has finite energy exactly when $c(u)=0$.

\section{The Kre\u\i n Formula and the Lagrangian Normalization}\label{sec:steklov_weighted}

\subsection{Classification of the symmetric realizations}\label{subsec:asymptotic_norm}

Lemma~\ref{lem:structure}(d) says that a boundary function has an $N$-parameter family of admissible harmonic extensions, so the Steklov operator is multi-valued until a selection rule is fixed. The next theorem classifies all linear selections whose operator is symmetric.

\begin{theorem}[Classification]\label{thm:krein}\label{prop:augmented_wellposed}
Let $f\mapsto u_f$ be a linear map $C^\infty(\partial\Sigma)\to\mathcal H$ with $u_f|_{\partial\Sigma}=f$, and let $Sf=\partial_\nu u_f|_{\partial\Sigma}$. Then $u_f=u^0_f+\sum_I(Gf)_{\ell,m}w_{\ell,m}$ for a unique linear $G\colon C^\infty(\partial\Sigma)\to\mathbb R^N$, and $S=S_0+\Psi G$. Moreover $S$ is symmetric on $C^\infty(\partial\Sigma)$ if and only if $G=M\Psi^\top$ for a symmetric $N\times N$ matrix $M$, which is then unique:
\begin{equation}\label{eq:krein}
S=S_M:=S_0+\Psi M\Psi^\top .
\end{equation}
Conversely every symmetric $M$ arises this way.
\end{theorem}

\begin{proof}
The first assertion is Lemma~\ref{lem:structure}(d) with $Gf=c(u_f)$. If $G=M\Psi^\top$ with $M$ symmetric then $\Psi M\Psi^\top$ is symmetric, and $S_0$ is symmetric by~\ref{F3}. Conversely let $S$ be symmetric; then $B:=\Psi G=S-S_0$ is symmetric on $C^\infty(\partial\Sigma)$ with range in $Y:=\operatorname{ran}\Psi$. For $f\in C^\infty\cap Y^\perp$ and all $h\in C^\infty$, $\langle Bf,h\rangle=\langle f,Bh\rangle=0$, so $Bf=0$. With $P=\Psi\Gamma^{-1}\Psi^\top$ the orthogonal projection onto $Y$ this gives $B=BP$ and $PB=B$. The matrix $\Psi^\top B\Psi$ has entries $\langle B\psi_j,\psi_i\rangle$ and is symmetric; set $M=\Gamma^{-1}\Psi^\top B\Psi\,\Gamma^{-1}$. Then $\Psi M\Psi^\top=PBP=B$, and $G=M\Psi^\top$ by injectivity of $\Psi$. Uniqueness follows from injectivity of $\Psi$ and surjectivity of $\Psi^\top$.
\end{proof}

\subsection{The symplectic form on the asymptotic data}

The classification above is stated in terms of a matrix. We now record the equivalent description in terms of the asymptotic data, which is the formulation used in Sections~\ref{sec:multiple_ends}--\ref{sec:higher_codim} and is intrinsic.

\begin{definition}\label{def:symplectic}
For $u,v\in\mathcal H$ set
\begin{equation}\label{eq:symplectic_form}
\omega(u,v)=\lim_{\rho\to0}\ \sum_\ell\int_{\{|\zeta_\ell|=\rho\}}\bigl(v\,\partial_n u-u\,\partial_n v\bigr)\,ds,
\end{equation}
where $n$ points towards the puncture. The integrand is a conformally invariant $1$-form, so the expression does not depend on the conformal coordinate, and each term is independent of $\rho$.
\end{definition}

\begin{definition}\label{def:full_asymptotic_data}\label{def:extraction}
For $u\in\mathcal H$ let $v_\ell$ be its regular part at $p_\ell$ (Lemma~\ref{lem:structure}(a)) and define $d_\ell(u)\in\mathbb R^{2N_\ell+1}$ by
\[
d_{\ell,0}(u)=v_\ell(p_\ell),\qquad
d_{\ell,2j-1}(u)+i\,d_{\ell,2j}(u)=\tfrac{2}{(j-1)!}\,\partial_\zeta^{\,j}v_\ell(p_\ell)\quad(1\leq j\leq N_\ell).
\]
The \emph{asymptotic data} of $u$ is $q(u)=(c(u),d(u))$.
\end{definition}

\begin{definition}\label{def:asymptotic_data_space}
The \emph{asymptotic data space} is $\mathcal A=\mathbb R^N\times\mathbb R^N$ with the bilinear form
\begin{equation}\label{eq:omega_explicit}
\omega\bigl((c,d),(c',d')\bigr)=2\pi\bigl(\langle c,d'\rangle-\langle d,c'\rangle\bigr).
\end{equation}
\end{definition}

\begin{proposition}\label{prop:A_symplectic}\label{prop:omega_explicit}\label{prop:boundary_formulation}
\begin{enumerate}[label=(\roman*)]
\item For $u,w\in\mathcal H$,
\begin{equation}\label{eq:boundary_pairing}
\int_{\partial\Sigma}\bigl(w\,\partial_\nu u-u\,\partial_\nu w\bigr)\,ds=-\,\omega\bigl(q(u),q(w)\bigr).
\end{equation}
The form~\eqref{eq:omega_explicit} is non-degenerate, hence symplectic, and by~\eqref{eq:boundary_pairing} it agrees with Definition~\ref{def:symplectic}.
\item $d(u^0_f)=-(2\pi)^{-1}\Psi^\top f$; $q(w_{\ell,m})=(e_{\ell,m},Qe_{\ell,m})$ with $Q$ a symmetric $N\times N$ matrix; $q(K)=\{(c,Qc)\}$ is Lagrangian; and $q(\mathcal H)=\mathcal A$.
\end{enumerate}
\end{proposition}

\begin{proof}
(i) Green's second identity on $\Sigma_\rho=\Sigma\setminus\bigcup_\ell\zeta_\ell^{-1}(D_\rho)$ gives $\int_{\partial\Sigma}(w\partial_\nu u-u\partial_\nu w)\,ds+\sum_\ell J_\ell=0$ with $J_\ell=\int_{|\zeta_\ell|=\rho}(w\partial_nu-u\partial_nw)\,ds$. Each $J_\ell$ is independent of $\rho$ and the integrand is conformally invariant, so we compute in the flat coordinate $\zeta=\rho e^{i\phi}$ with $n=-\partial_\rho$ and $ds=\rho\,d\phi$. Write $q(u)=(c,d)$, $q(w)=(C,D)$, and the regular part of $u$ as $v=\operatorname{Re}\sum_{j\geq0}h_j\zeta^j$; then $h_0=d_0$ and $jh_j=d_{2j-1}+id_{2j}$ for $1\leq j\leq N_\ell$. On each circle the modes $\cos j\phi,\sin j\phi$ are orthogonal for different $j$. The mode $j=0$ contributes $2\pi(c_0D_0-d_0C_0)$. For $1\leq j\leq N_\ell$ the $\cos j\phi$-coefficient of $u$ on $|\zeta|=\rho$ is $a=c_{2j-1}\rho^{-j}+(d_{2j-1}/j)\rho^{j}$, and if $A$ is that of $w$ then $\pi\rho\,(A\,\partial_n a-a\,\partial_n A)=2\pi(c_{2j-1}D_{2j-1}-d_{2j-1}C_{2j-1})$, the powers $\rho^{\pm2j-1}$ cancelling; the $\sin j\phi$-coefficient is $-(c_{2j}\rho^{-j}+(d_{2j}/j)\rho^{j})$, and the sign cancels in the product. These are the terms of~\eqref{eq:omega_explicit}. The modes $j>N_\ell$ are purely regular and have vanishing Wronskian, and all remaining pairings vanish by orthogonality.

(ii) Define $Q$ by $Qe_{\ell,m}:=d(w_{\ell,m})$, so that $q(w_{\ell,m})=(e_{\ell,m},Qe_{\ell,m})$ by Lemma~\ref{lem:structure}(c). The first identity is~\eqref{eq:boundary_pairing} with $u=u^0_f$, for which $c=0$, and $w=w_{\ell,m}$, which vanishes on $\partial\Sigma$: it reads $-\langle f,\psi_{\ell,m}\rangle=2\pi\,d_{\ell,m}(u^0_f)$. Symmetry of $Q$ follows from~\eqref{eq:boundary_pairing} and $w_{\ell,m}|_{\partial\Sigma}=0$. Then $q(K)=\{(c,Qc)\}$ is isotropic of dimension $N$, hence Lagrangian, and by Lemma~\ref{lem:structure}(d)
\begin{equation}\label{eq:green_step5}
q\Bigl(u^0_f+\sum c_{\ell,m}w_{\ell,m}\Bigr)=\bigl(c,\ -(2\pi)^{-1}\Psi^\top f+Qc\bigr),
\end{equation}
and $\Psi^\top$ is onto, so $q(\mathcal H)=\mathcal A$.
\end{proof}

\begin{remark}\label{rem:extension_theory}
By part (ii) the data map is onto: every prescribed pair of growing and regular data is realized by an admissible harmonic function.
\end{remark}

\subsection{Lagrangian normalizations}\label{sec:steklov_def}

\begin{definition}\label{def:lagrangian_subspace}\label{def:lagrangian_norm}
A subspace $\mathcal L\subset\mathcal A$ is a \emph{normalization} if $\mathcal A=q(K)\oplus\mathcal L$. Given such an $\mathcal L$, the \emph{normalized extension} of $f$ is the $u_f\in\mathcal H$ with trace $f$ and $q(u_f)\in\mathcal L$, which exists and is unique by Proposition~\ref{prop:boundary_normalization}, and the associated Steklov operator is $\mathcal S_{\mathcal L}f=\partial_\nu u_f|_{\partial\Sigma}$.
\end{definition}

\begin{proposition}[Lagrangian dictionary]\label{prop:boundary_normalization}\label{prop:coordinate_lagrangian}
Let $\mathcal L\subset\mathcal A$ be a subspace. Every $f\in C^\infty(\partial\Sigma)$ has exactly one admissible extension with data in $\mathcal L$ if and only if $\mathcal A=q(K)\oplus\mathcal L$. In that case $\mathcal S_{\mathcal L}$ is symmetric if and only if $\mathcal L$ is Lagrangian, and then $\mathcal S_{\mathcal L}=S_M$ with
\begin{equation}\label{eq:lagrangian_boundary}
\mathcal L=\mathcal L_M:=\{(c,d):\ c=-2\pi M(d-Qc)\}.
\end{equation}
The map $M\mapsto\mathcal L_M$ is a bijection from symmetric $N\times N$ matrices onto the Lagrangian complements of $q(K)$.
\end{proposition}

\begin{proof}
The map $\tau(c,d)=(c,d-Qc)$ is symplectic because $Q$ is symmetric, and sends $q(K)$ to $\mathbb R^N\times0$. By~\eqref{eq:green_step5} the $\tau$-data of the extensions of $f$ fill the affine space $\mathbb R^N\times\{d_f\}$ with $d_f=-(2\pi)^{-1}\Psi^\top f$, and $d_f$ ranges over $\mathbb R^N$. So $\tau\mathcal L$ meets every such affine space exactly once if and only if it is the graph $\{(Zd,d)\}$ of a linear map $Z$, that is, if and only if $\mathcal A=q(K)\oplus\mathcal L$. Then $c(u_f)=Zd_f=M\Psi^\top f$ with $M=-Z/(2\pi)$, so $\mathcal S_{\mathcal L}=S_0+\Psi M\Psi^\top$, which by Theorem~\ref{thm:krein} is symmetric exactly when $M$ is. Since $\omega((Zd,d),(Zd',d'))=2\pi(\langle Zd,d'\rangle-\langle d,Zd'\rangle)$, that happens exactly when $\operatorname{graph}Z$, equivalently $\mathcal L$, is Lagrangian. Finally $\mathcal L=\tau^{-1}\operatorname{graph}(-2\pi M)=\mathcal L_M$.
\end{proof}

\begin{remark}\label{rem:why_lagrangian}
Non-Lagrangian complements give $S_0+\Psi M\Psi^\top$ with $M$ non-symmetric, which is not symmetric; Lagrangian subspaces meeting $q(K)$ non-trivially do not define an extension for every $f$. So ``Lagrangian complement'' is exactly the condition for a symmetric realization defined on all of $C^\infty(\partial\Sigma)$.
\end{remark}

\begin{remark}[Comparison with pseudo-Laplacians]\label{rem:pseudo_laplacians}
The construction is the Steklov counterpart of a classical picture for the Laplacian itself. On a compact surface $X$ with a marked point $P$, the Laplacian with domain $C^\infty_c(X\setminus P)$ has deficiency indices $(1,1)$. Its self-adjoint extensions, the pseudo-Laplacians of Colin de Verdi\`ere~\cite{ColinDeVerdiere1982}, are fixed by a linear relation between the coefficient of $\log r$ and the constant term at $P$, that is, by a Lagrangian line in a two-dimensional symplectic space of asymptotic data; the Friedrichs extension, in which the logarithmic coefficient vanishes, is the Laplacian of $X$; and the resolvents of two extensions differ by a rank-one operator, the Kre\u\i n formula~\cite[Section~2]{AHK2012}. On surfaces with conical singularities the space of asymptotic data grows with the cone angles, the extensions again correspond to Lagrangian subspaces, and the Friedrichs extension to the vanishing of the singular coefficients~\cite{HillairetKokotov2013}.

In the present setting $(\mathcal A,\omega)$ plays the part of the space of asymptotic data with its boundary form; the bounded realization $S_0$, whose extensions have no principal parts, plays the part of the Friedrichs extension; and Theorem~\ref{thm:krein} is the Kre\u\i n formula, the perturbation having rank at most $N$. At each puncture the pair $(c_{\ell,0},d_{\ell,0})$, the coefficient of $-\log|\zeta_\ell|$ and the value of the regular part, is exactly the asymptotic datum of a pseudo-Laplacian, with the form $2\pi(c\,d'-d\,c')$. Two differences are specific to the Steklov problem. First, the operator acts on $L^2(\partial\Sigma)$, and the modes admitted at the punctures need not be square-integrable, so no deficiency count fixes $N$: the orders $N_\ell$ record how much growth is admitted, and $N_\ell=1$ is what the coordinate functions require (Remark~\ref{rem:orders}). Second, the cone angle plays no role: harmonicity and the Dirichlet energy are conformally invariant, so the multiplicity of an end, which is its cone angle at infinity divided by $2\pi$, changes neither $\mathcal A$ nor $\omega$ (Section~\ref{subsec:codim_intrinsic}).
\end{remark}

\section{Discreteness}\label{sec:proof_discrete}

\begin{theorem}[Spectral properties]\label{thm:spectral}\label{prop:principal_symbol}\label{prop:symmetry}\label{cor:S_elliptic}
Let $M$ be symmetric. The closure of $S_M$ is self-adjoint on $L^2(\partial\Sigma)$ with domain $H^1(\partial\Sigma)$. It is a classical elliptic pseudodifferential operator of order one with principal symbol $|\xi|$ and has compact resolvent. Its eigenvalues $\sigma_1\leq\sigma_2\leq\cdots$ tend to $+\infty$ and $\sigma_1\geq-\|M\|\,\|\Gamma\|$. With $\mathcal N_M(\tau)=\#\{j:\sigma_j(S_M)<\tau\}$ and $n_\pm(M)$ the numbers of positive and negative eigenvalues of $M$,
\begin{equation}\label{eq:counting}
\mathcal N_0(\tau)-n_+(M)\ \leq\ \mathcal N_M(\tau)\ \leq\ \mathcal N_0(\tau)+n_-(M)\qquad(\tau\in\mathbb R).
\end{equation}
In particular $S_M$ has at most $n_-(M)$ negative eigenvalues, counted with multiplicity.
\end{theorem}

\begin{proof}
$B=\Psi M\Psi^\top$ is bounded and symmetric on $L^2$, with smooth kernel $\sum M_{ij}\psi_i(x)\psi_j(y)$ and $\|B\|\leq\|M\|\|\Psi\|^2=\|M\|\|\Gamma\|$. Hence $S_M=S_0+B$ on $H^1$. It is self-adjoint because $B$ is a bounded symmetric perturbation, and it is a pseudodifferential operator with the same principal symbol by~\ref{F3}, $B$ being smoothing. The resolvent identity $(S_M-\lambda)^{-1}=(S_0-\lambda)^{-1}\bigl(\mathrm{Id}-B(S_M-\lambda)^{-1}\bigr)$ shows the resolvent is compact. Since $S_0\geq0$ we get $\langle S_Mf,f\rangle\geq-\|B\|\|f\|^2$, so the spectrum is discrete, bounded below and infinite, whence $\sigma_j\to+\infty$.

For~\eqref{eq:counting} use Glazman's lemma: for a self-adjoint $A$ with compact resolvent, $\mathcal N_A(\tau)$ is the largest dimension of a subspace $V\subset\operatorname{Dom}A$ on which $\langle Af,f\rangle<\tau\|f\|^2$ for $f\neq0$. Let $V\subset H^1$ realize $\mathcal N_M(\tau)$ and put $W=\{f:\Psi^\top f\in E_{\geq0}(M)\}$, where $E_{\geq0}(M)$ is the non-negative spectral subspace of $M$. Then $W$ has codimension at most $n_-(M)$ and on $W$ we have $\langle S_0f,f\rangle\leq\langle S_Mf,f\rangle$. Hence $\langle S_0f,f\rangle<\tau\|f\|^2$ on $V\cap W\setminus0$ and $\mathcal N_0(\tau)\geq \mathcal N_M(\tau)-n_-(M)$. Exchanging the roles of $S_0$ and $S_M$, that is, replacing $M$ by $-M$, gives the other inequality. Since $\mathcal N_0(0)=0$, the last claim follows.
\end{proof}

\begin{corollary}\label{cor:bounded_below}\label{prop:finite_rank}
Every Lagrangian normalization yields a self-adjoint Steklov operator with compact resolvent, discrete real spectrum bounded below and accumulating only at $+\infty$; the number of negative eigenvalues is at most $\operatorname{rank}M$, and the eigenvalue asymptotics of the compact surface $(\overline\Sigma,\hat g)$ \cite{GPPS2014} transfer to it up to an index shift of at most $\operatorname{rank}M$.
\end{corollary}

\begin{proof}
Combine Proposition~\ref{prop:boundary_normalization} with Theorem~\ref{thm:spectral}; the counting bound~\eqref{eq:counting} gives the index shift.
\end{proof}

\subsection{Coordinate Functions as Eigenfunctions}\label{subsec:coord_eigen}

We now specialize to an exterior FBMS and recover the coordinate eigenvalue $-1$.

\begin{lemma}\label{lem:coord_harmonic}
Let $\Sigma$ be an exterior FBMS and $\ell_a=\langle a,\mathbf x\rangle$, $f_a=\ell_a|_{\partial\Sigma}$. Then $\Delta_g\ell_a=0$ and $\partial_\nu\ell_a=-f_a$ on $\partial\Sigma$.
\end{lemma}

\begin{proof}
Minimality gives $\Delta_g\mathbf x=0$. Along $\partial\Sigma$, $|\mathbf x|\equiv1$ gives $\mathbf x\perp d\mathbf x(T\partial\Sigma)$; since $\mathbf x$ is a unit tangent vector, $\mathbf x=\pm d\mathbf x(\nu)$. Because $|\mathbf x|^2\geq1$ on $\Sigma$ with equality on $\partial\Sigma$, $0\geq\partial_\nu|\mathbf x|^2=2\langle\mathbf x,d\mathbf x(\nu)\rangle$, so $d\mathbf x(\nu)=-\mathbf x$ and $\partial_\nu\ell_a=-\ell_a$.
\end{proof}

Let $X=\{a\in\mathbb R^n:\ell_a\in\mathcal H\}$ be the admissible directions, $X_0=\{a\in X:\ell_a\equiv0\}$, $F=\{f_a:a\in X\}$ and $k=\dim F$. Writing $c_a=c(\ell_a)$, Lemmas~\ref{lem:structure}(d) and~\ref{lem:coord_harmonic} give
\begin{equation}\label{eq:coord_relation}
-f_a=\partial_\nu\ell_a=S_0f_a+\Psi c_a\qquad(a\in X).
\end{equation}

\begin{lemma}[Symmetric extension]\label{lem:symmetric_extension}
Let $Y\subset\mathbb R^N$ be a subspace and $T\colon Y\to\mathbb R^N$ linear with $\langle Ty,y'\rangle=\langle y,Ty'\rangle$ on $Y$. The symmetric matrices $M$ with $M|_Y=T$ form a non-empty affine space of dimension $\tfrac12(N-\dim Y)(N-\dim Y+1)$.
\end{lemma}

\begin{proof}
Let $P$ be the orthogonal projection onto $Y$, $\tilde T=TP$ and $M=\tilde T+\tilde T^\top-P\tilde TP$. The last term is symmetric by hypothesis, so $M$ is symmetric; for $y\in Y$ and $z\in\mathbb R^N$, $\langle\tilde T^\top y,z\rangle=\langle y,TPz\rangle=\langle PTy,z\rangle$, so $My=Ty$. Two solutions differ by a symmetric matrix vanishing on $Y$, that is, by a symmetric form on $Y^\perp$.
\end{proof}

\begin{proposition}[Coordinate realizations]\label{prop:coordinate_eigenfunctions}\label{thm:transversality}\label{cor:coord_lagrangian}
\begin{enumerate}[label=(\roman*)]
\item For $a\in X$: $f_a=0\iff\ell_a\equiv0\iff\Psi^\top f_a=0$. Hence $k=\dim X-\dim X_0$ and $\Psi^\top|_F$ is injective.
\item For symmetric $M$: $S_Mf_a=-f_a$ for all $a\in X$ if and only if $M\Psi^\top f_a=c_a$ for all $a\in X$. Such $M$ exist and form an affine space of dimension $\tfrac12(N-k)(N-k+1)$.
\item For every such $M$, $\dim\ker(S_M+1)\geq k$, and $S_M$ has at most $n_-(M)$ negative eigenvalues. If $k=N$ then $M$ is unique and the negative spectrum of $S_M$ is exactly $\{-1\}$, with multiplicity $N$.
\item Equivalently, $V=\{q(\ell_a):a\in X\}$ is isotropic of dimension $k$ with $V\cap q(K)=0$, and the Lagrangian complements of $q(K)$ containing $V$ are exactly the $\mathcal L_M$ of (ii). If $k=N$ then $V$ itself is the unique one and $\mathcal A=q(K)\oplus V$.
\end{enumerate}
\end{proposition}

\begin{proof}
(i) Pairing~\eqref{eq:coord_relation} with $f_a$ gives $-\|f_a\|^2=\langle S_0f_a,f_a\rangle+\langle c_a,\Psi^\top f_a\rangle$. If $\Psi^\top f_a=0$ then $-\|f_a\|^2=\langle S_0f_a,f_a\rangle\geq0$, so $f_a=0$; and $f_a=0$ forces $\partial_\nu\ell_a=0$, hence $\ell_a\equiv0$ by Lemma~\ref{lem:UCP}.

(ii) By~\eqref{eq:coord_relation}, $S_Mf_a=-f_a$ iff $\Psi M\Psi^\top f_a=\Psi c_a$, iff $M\Psi^\top f_a=c_a$. Let $Y=\Psi^\top F$, of dimension $k$ by (i), and set $T(\Psi^\top f_a)=c_a$; this is well defined because $\Psi^\top f_a=\Psi^\top f_b$ forces $\ell_{a-b}\equiv0$, and symmetric because $\langle T\Psi^\top f_a,\Psi^\top f_b\rangle=\langle\Psi c_a,f_b\rangle=-\langle f_a,f_b\rangle-\langle S_0f_a,f_b\rangle$ is symmetric in $(a,b)$. Apply Lemma~\ref{lem:symmetric_extension}.

(iii) The eigenspace of $-1$ contains $F$; the bound is Theorem~\ref{thm:spectral}. If $k=N$ the affine space in (ii) is a point and $N\leq\#\{\sigma_j<0\}\leq n_-(M)\leq N$.

(iv) By~\eqref{eq:boundary_pairing} and Lemma~\ref{lem:coord_harmonic}, $\omega(q(\ell_a),q(\ell_b))=-\int_{\partial\Sigma}(-f_bf_a+f_af_b)\,ds=0$, so $V$ is isotropic. By~\eqref{eq:green_step5}, $d(\ell_a)-Qc_a=-(2\pi)^{-1}\Psi^\top f_a$, so $q(\ell_a)\in q(K)$ iff $\Psi^\top f_a=0$ iff $\ell_a\equiv0$ by (i); hence $V\cap q(K)=0$ and $\dim V=k$. Finally $q(\ell_a)\in\mathcal L_M$ is exactly condition (ii).
\end{proof}

\begin{theorem}\label{thm:main_body}
Let $\Sigma\subset\mathbb R^3\setminus B^3$ be a properly embedded exterior FBMS with compact boundary and one regular end, and take $N_1=1$, so $N=3$. Then every Lagrangian normalization gives a self-adjoint Steklov operator with compact resolvent, discrete spectrum bounded below and tending to $+\infty$. Unless $\Sigma$ lies in a plane through the origin, $k=3=N$: there is exactly one normalization for which the coordinate traces are eigenfunctions with eigenvalue $-1$, and for it the negative spectrum is exactly $\{-1\}$ with multiplicity $3$.
\end{theorem}

\begin{proof}
The hypotheses of Section~\ref{sec:compactification} hold by Lemma~\ref{lem:conformal_end}, and $X=\mathbb R^3$ with $N_1=1$ by Remark~\ref{rem:orders}. The first claim is Corollary~\ref{cor:bounded_below}. Since $X=\mathbb R^3$, by definition $X_0\neq0$ exactly when $\Sigma\subset a^\perp$ for some $a\neq0$, that is, when $\Sigma$ lies in a plane through the origin; otherwise Proposition~\ref{prop:coordinate_eigenfunctions}(i) gives $k=3-\dim X_0=3=N$, and Proposition~\ref{prop:coordinate_eigenfunctions}(iii) applies.
\end{proof}

\begin{remark}\label{rem:robin_link}
The operator $S_M+\mathrm{Id}$ annihilates the coordinate traces, so a coordinate-compatible realization is exactly a Robin-type realization for which the coordinate functions are harmonic with vanishing Robin data.
\end{remark}

\begin{remark}\label{rem:sigma_minus_one}
For one end, with $\Sigma$ not in a plane through the origin (so that $k=N$), the quantity $|\sigma_{-1}|\operatorname{Length}(\partial\Sigma)$ proposed in the introduction equals $\operatorname{Length}(\partial\Sigma)$, since $\sigma_{-1}=-1$ exactly. It is therefore not an independent functional in the one-ended case; with several ends the coordinate-compatible normalization is no longer unique and the quantity can depend on the choice.
\end{remark}

\section{Applications and Examples}\label{sec:applications}

\subsection{\texorpdfstring{The Mazet--Mendes family $C_\alpha$}{The Mazet-Mendes family C\_alpha}}\label{subsec:mazet_mendes}

The rotational family $C_\alpha$, $\alpha\in[0,\pi/2)$, has boundary at height $\sin\alpha$ \cite[Section~4]{MazetMendes2022}. Here $C_0=\{z=0\}\setminus B^3$ is the exterior of the unit disk in a plane --- a planar exterior domain, not a disk --- while $\alpha>0$ gives a catenoidal end.

\begin{example}\label{ex:catenoid_family}
For $\alpha>0$, Theorem~\ref{thm:main_body} gives discrete spectrum with eigenvalue $-1$. Stability changes at $\alpha=\pi/4$ \cite[Proposition~4.1]{MazetMendes2022}, but the Steklov conclusion does not. By \cite[Theorem~1.2]{MazetMendes2022}, these surfaces and $C_0$ exhaust, up to rotation, the properly embedded one-regular-end class. Separation of variables on the end gives the spectrum of the unique coordinate realization: $\{-1,-1,-1\}\cup\{n/\cos\alpha:\ n\geq2\}$, each positive eigenvalue with multiplicity two, where $\cos\alpha$ is the radius of the boundary circle; in particular $0$ is not an eigenvalue.
\end{example}

\begin{example}\label{ex:flat_disk}
The end of $C_0$ is planar, but nothing in Sections~\ref{sec:compactification}--\ref{sec:proof_discrete} distinguishes planar from catenoidal ends (Remark~\ref{rem:harmonic_planar}), so the construction applies unchanged. Here $V=\operatorname{span}\{q(x),q(y)\}$ is two-dimensional, so $N-k=1$ and the coordinate-compatible normalizations form a one-parameter family. The member $q(u)\in\operatorname{span}\{q(1),q(x),q(y)\}$ is equivalently obtained by adding zero logarithmic flux to the two coordinate constraints. It selects $r\cos\theta,r\sin\theta$ and gives
\[
\{-1, -1,\; 0,\; 2, 2,\; 3, 3,\; \ldots\}.
\]
Here $-1$ comes from the two coordinates, $0$ from constants, and $n\geq2$ from the decaying modes, with the displayed multiplicities.
\end{example}

\begin{remark}\label{rem:comparison_BGGLP}
On $C_0$, the bounded far-field normalization of \cite{BGGLP2026} selects $r^{-1}\cos\theta$ and eigenvalue $1$, whereas $\omega(u,x_1)=0$ selects $r\cos\theta$ and eigenvalue $-1$; their pairings with $x_1$ are $-2\pi$ and $0$. For $n\geq2$ both realizations select the decaying mode, while constants give zero. Thus
\begin{align*}
\text{Far-field normalization \cite{BGGLP2026}:}& \quad \{0,\; 1, 1,\; 2, 2,\; 3, 3,\; \ldots\}, \\
\text{Lagrangian normalization (ours):}& \quad \{-1, -1,\; 0,\; 2, 2,\; 3, 3,\; \ldots\}.
\end{align*}
The discrepancy is a change of self-adjoint realization, not a contradiction.
\end{remark}

\section{Several Ends}\label{sec:multiple_ends}\label{subsec:transversality}\label{subsec:multi_scope}

Sections~\ref{sec:compactification}--\ref{sec:proof_discrete} were stated for $L\geq1$ ends throughout, so the multi-end case needs no new analysis. By Lemma~\ref{lem:conformal_end} a surface with compact boundary and finitely many regular ends has one punctured conformal disk per end; the compactification fills all punctures at once, $N=\sum_\ell(2N_\ell+1)$, and Theorem~\ref{thm:krein}, Theorem~\ref{thm:spectral} and Proposition~\ref{prop:coordinate_eigenfunctions} apply verbatim. Catenoidal and planar ends are treated identically; the type of an end reappears only through which coordinate functions have a principal part (Remark~\ref{rem:harmonic_planar}).

\begin{theorem}\label{thm:multiple_ends}
Let $\Sigma$ be an exterior FBMS with compact boundary and finitely many regular ends, and take $N_\ell=1$ for every $\ell$, so $N=3L$. Then:
\begin{enumerate}[label=(\roman*)]
\item every Lagrangian complement $\mathcal L$ of $q(K)$ gives a self-adjoint Steklov operator $\mathcal S_{\mathcal L}$ with compact resolvent, discrete real spectrum bounded below and accumulating only at $+\infty$, and with at most $\operatorname{rank}M$ negative eigenvalues, where $\mathcal L=\mathcal L_M$;
\item the coordinate-compatible normalizations form a non-empty affine family, and for each of them every non-zero coordinate trace is an eigenfunction with eigenvalue $-1$; all three traces are non-zero unless $\Sigma$ lies in a plane through the origin;
\item the eigenvalue asymptotics of the compact surface $(\overline\Sigma,\hat g)$ transfer to every such realization up to an index shift of at most $\operatorname{rank}M$.
\end{enumerate}
\end{theorem}

\begin{proof}
(i) is Corollary~\ref{cor:bounded_below} together with Proposition~\ref{prop:boundary_normalization}; (ii) is Proposition~\ref{prop:coordinate_eigenfunctions}(i)--(iii) with Lemma~\ref{lem:conformal_end} supplying the hypotheses of Section~\ref{sec:compactification}; (iii) is the counting bound~\eqref{eq:counting}.
\end{proof}

What the ends change is only the size of the family. The coordinate data $V$ have dimension $k\leq3$, with $k=3$ unless $\Sigma$ lies in a plane through the origin (Proposition~\ref{prop:coordinate_eigenfunctions}(i)), while $N=3L$. For one end $k=N$ and the normalization is unique (Theorem~\ref{thm:main_body}); for $L\geq2$ the coordinate-compatible normalizations form an affine family of dimension $\tfrac12(N-k)(N-k+1)>0$, and the coordinates alone single out no member. Every realization differs from the compact Dirichlet-to-Neumann map of $(\overline\Sigma,\hat g)$ by a symmetric operator of rank at most $N$, so $N$ also bounds the number of negative eigenvalues. The choice of orders $(N_\ell)$ is a modelling decision; $N_\ell=1$ retains exactly the coordinate functions. No \emph{connected} exterior FBMS in $\mathbb R^3$ with $L\geq2$ regular ends is known (disjoint unions, such as $C_\alpha$ together with its mirror image in $\{z=0\}$, trivially have $L=2$); Theorem~\ref{thm:multiple_ends} describes connected ones should they exist.

\section{Intrinsic Generalization and Higher Codimension}\label{sec:higher_codim}

In Sections~\ref{sec:proof_discrete} and~\ref{sec:multiple_ends} we worked with exterior minimal surfaces in $\mathbb{R}^3 \setminus B^3$. We now separate the intrinsic analytic theorem from its geometric applications:
\begin{itemize}
    \item The \textbf{intrinsic discreteness theorem} (compact resolvent, self-adjointness, and discrete spectrum bounded below) applies to any Riemannian surface with compact boundary and finitely many intrinsic regular ends, at any choice of orders $(N_\ell)$ (Theorem~\ref{thm:codim_discrete}). It requires neither an ambient immersion nor minimality.
    \item For an exterior FBMS $\Sigma^2 \subset \mathbb{R}^n \setminus B^n$, the \textbf{coordinate-eigenfunction statement} ($\sigma = -1$ with multiplicity inherited from the immersion) applies to any \emph{admissible} subset of ambient coordinates --- those whose restrictions to $\Sigma$ lie in $\mathcal H$ for the chosen orders (Theorem~\ref{thm:codim_coord_eigen}).
\end{itemize}
This separation is important: as the example of the Fraser--Sargent exterior surfaces $EFS_{k,l} \subset \mathbb{R}^4 \setminus B^4$ \cite{MedvedevMorozov2022} shows (Section~\ref{subsec:FS_example}), not all ambient coordinates need lie in $\mathcal H$, even in the simplest higher-codimension examples. As in Section~\ref{sec:multiple_ends}, $L$ denotes the number of ends.

\subsection{Intrinsic analysis does not require an immersion}\label{subsec:codim_intrinsic}

The spectral machinery of Sections~\ref{sec:compactification}--\ref{sec:multiple_ends} is \emph{intrinsic}: the compactification and the space $\mathcal H$ are defined by the conformal structure alone; the analytic input is the Dirichlet-to-Neumann map of a compact surface together with a bounded smoothing correction; and the finite-rank decomposition uses the intrinsic Dirichlet kernel. End multiplicity does not enter at all: a punctured disk is a punctured disk whatever the
conformal factor, so the multiplicity changes neither $\mathcal H$ nor $\mathcal A$, and shows up
only in where the ambient coordinate functions sit. None of these ingredients refers to an ambient immersion or to minimality.

Nor does the intrinsic theory require orientability: the standing convention of Section~\ref{sec:notation} is used only in the extrinsic hypersurface discussion. The ends are annuli and hence orientable regardless of the global topology.

In particular, the exterior half $EFS_{2,1}$ considered below is an annulus; the M\"obius-band topology belongs to the full Fraser--Sargent surface and its compact interior portion (Section~\ref{subsec:FS_example}).

We first formulate the analytic end hypothesis abstractly. The definition of regular ends in Section~\ref{subsec:planar_ends} is extrinsic and specific to hypersurfaces in $\mathbb{R}^3$, while the intrinsic theory only needs the following metric structure at infinity. The geometric applications below reintroduce the ambient immersion and, as the Fraser--Sargent surfaces show, may involve a further discrete invariant: the end multiplicity.

\begin{definition}\label{def:intrinsic_end}
An end $E$ of a Riemannian surface $(\Sigma, g)$ is an \emph{intrinsic regular end of multiplicity $m \in \mathbb{N}$} if $E$ admits conformal coordinates $(t, \theta) \in [t_0, \infty) \times S^1$ in which
\[
g = \lambda(t, \theta)\,(dt^2 + d\theta^2), \qquad \lambda = c\,e^{2mt}\bigl(1 + O(e^{-\epsilon t})\bigr)
\]
for some constants $c, \epsilon > 0$, the error term being preserved by $(t,\theta)$-derivatives. We write $r = e^t$ for the conformal radial coordinate.
\end{definition}

Definition~\ref{def:intrinsic_end} is an explicit analytic hypothesis. Finite total curvature gives a conformal puncture and meromorphic Weierstrass data, but in arbitrary codimension we do not infer from it alone the derivative-controlled metric expansion required in the definition. The hypothesis holds for the regular multiplicity-one ends used in Section~\ref{sec:multiple_ends} and, by direct computation, for the ends of $EFS_{k,l} \subset \mathbb{R}^4 \setminus B^4$, where $m=k$ (Section~\ref{subsec:FS_example}).

An intrinsic regular end is a conformal punctured disk: in the coordinates of Definition~\ref{def:intrinsic_end} put $\zeta=e^{-(t+i\theta)}$, so that the part $\{t>t_0\}$ of $E$ is $\{0<|\zeta|<e^{-t_0}\}$, $\zeta\to0$ exactly when $t\to\infty$, and $g=\mu|d\zeta|^2$ with $\mu>0$ smooth. The hypotheses of Section~\ref{sec:compactification} therefore hold verbatim, and with them Theorems~\ref{thm:krein} and~\ref{thm:spectral}, Proposition~\ref{prop:boundary_normalization} and Proposition~\ref{prop:coordinate_eigenfunctions}. No immersion, orientability, completeness or curvature bound is used.

What the multiplicity controls is not the function space but the growth that a coordinate function can have. On an end of multiplicity $m$ the intrinsic radial variable is $r=e^t$, and the growing harmonic modes $r^{n}\cos n\theta$, $r^{n}\sin n\theta$, $n\in\mathbb N$, are the real and imaginary parts of the principal part $\zeta^{-n}$; hence retaining the modes of growth rate $n$ means
\begin{equation}\label{eq:multiplicity_admissibility}
n\leq N_\ell.
\end{equation}
The choice
\begin{equation}\label{eq:adapted_window}
N_\ell=1
\end{equation}
retains exactly the four non-decaying modes $1$, $\log r$, $r\cos\theta$, $r\sin\theta$, as for the regular ends in $\mathbb R^3$; a coordinate function growing like $r^{\,\mathrm{p}}$ requires $N_\ell\geq \mathrm{p}$.

\paragraph{Intrinsic asymptotic data.} The data map $q$, the symplectic form $\omega$ and the Lagrangian dictionary of Section~\ref{sec:steklov_weighted} are defined from the conformal chart alone, so they carry over unchanged, with $N=\sum_\ell(2N_\ell+1)$.

\begin{theorem}[Intrinsic discreteness]\label{thm:codim_discrete}
Let $(\Sigma,g)$ be a Riemannian surface with smooth compact boundary, every component of which meets $\partial\Sigma$, and with finitely many intrinsic regular ends. Fix orders $(N_\ell)$. Then every Lagrangian complement $\mathcal L$ of $q(K)$ in $\mathcal A$ determines a self-adjoint Steklov operator $\mathcal S_{\mathcal L}$ on $L^2(\partial\Sigma)$ with domain $H^1(\partial\Sigma)$, compact resolvent, and discrete real spectrum bounded below and accumulating only at $+\infty$. It is a classical elliptic pseudodifferential operator of order one with principal symbol $|\xi|$, and its counting function differs from that of the compactified Dirichlet-to-Neumann map by at most $\operatorname{rank}M$.
\end{theorem}

\begin{proof}
Each end is a conformal punctured disk, so Section~\ref{sec:compactification} applies; the conclusion is Theorem~\ref{thm:spectral} together with Proposition~\ref{prop:boundary_normalization}.
\end{proof}

\subsection{Coordinates in higher codimension}\label{subsec:codim_isotropic}

Let now $\Sigma^2\subset\mathbb R^n\setminus B^n$ be an exterior FBMS in arbitrary codimension, with the conventions of Definition~\ref{def:exterior_FBMS}, whose ends satisfy the hypotheses of Theorem~\ref{thm:codim_discrete}. As in Section~\ref{subsec:coord_eigen} put $\ell_a=\langle a,\mathbf x\rangle$, $f_a=\ell_a|_{\partial\Sigma}$, let $X=\{a\in\mathbb R^n:\ell_a\in\mathcal H\}$ be the admissible directions for the chosen orders $(N_\ell)$, $X_0=\{a\in X:\ell_a\equiv0\}$, $k=\dim X-\dim X_0$, and $V=\{q(\ell_a):a\in X\}$.

\begin{theorem}\label{thm:codim_coord_eigen}\label{prop:isotropy_codim}\label{thm:transversality_codim}\label{cor:coord_lagrangian_codim}\label{rem:VI_not_lagrangian}
Proposition~\ref{prop:coordinate_eigenfunctions} holds verbatim in $\mathbb R^n$ and at every order. Thus $V$ is isotropic of dimension $k$ and $V\cap q(K)=\{0\}$; the Lagrangian complements of $q(K)$ containing $V$ form a non-empty affine family of dimension $\tfrac12(N-k)(N-k+1)$; for each of them $\mathcal S_{\mathcal L}f_a=-f_a$ for every $a\in X$, so that $\dim\ker(\mathcal S_{\mathcal L}+1)\geq k$; and if $k=N$ the normalization is unique and its negative spectrum is exactly $\{-1\}$, of multiplicity $N$. In particular $V$ is Lagrangian exactly when $k=N$, which fails as soon as $N>n$.
\end{theorem}

\begin{proof}
The proof of Proposition~\ref{prop:coordinate_eigenfunctions} uses only $\Delta_g\ell_a=0$ and $\partial_\nu\ell_a=-\ell_a$ for $a\in X$, which hold in every codimension by Lemma~\ref{lem:coord_harmonic}, and the results of Sections~\ref{sec:compactification}--\ref{sec:proof_discrete}, which hold at every order (Theorem~\ref{thm:codim_discrete}). The last sentence follows from $k\leq\dim X\leq n$.
\end{proof}

At a given order a coordinate-compatible realization contains only the admissible coordinates; enlarging $(N_\ell)$ enlarges $X$ and can bring further coordinates in, at the cost of enlarging $N$ and hence the family of realizations. By~\eqref{eq:counting} a realization with $M\geq0$, such as the bounded one $M=0$, has no negative spectrum at all.

\subsection{Example: Fraser--Sargent exterior surfaces}\label{subsec:FS_example}

The Fraser--Sargent exterior surfaces $EFS_{k,l} \subset \mathbb{R}^4 \setminus B^4$ (parameters $k > l \geq 1$ relatively prime) of Medvedev--Morozov \cite{MedvedevMorozov2022} are one-ended $2$-surfaces in $\mathbb{R}^4$. In the conformal coordinates $(t, \theta)$ of the end, the induced metric is $g = \lambda(t)(dt^2 + d\theta^2)$ with $\lambda(t) = c\,e^{2kt}\bigl(1 + O(e^{-2(k-l)t})\bigr)$, so by Definition~\ref{def:intrinsic_end} the end is an intrinsic regular end of multiplicity $m = k$; in the intrinsic conformal radial coordinate $r = e^t$ --- \emph{not} the ambient radius, which has order $r^k$ --- the coordinate pairs $(x_1,x_2)$ and $(x_3,x_4)$ have leading orders $r^l$ and $r^k$, respectively. Admissibility is governed by the order alone: by~\eqref{eq:multiplicity_admissibility} the growing modes $r^{n}\cos n\theta$, $r^{n}\sin n\theta$ are retained precisely when $n\leq N_1$, and the multiplicity $m=k$ plays no part in it. At the order~\eqref{eq:adapted_window}, $N_1=1$, the admitted growing modes are exactly $1$, $\log r$, $r\cos\theta$ and $r\sin\theta$, as in the $\mathbb{R}^3$ analysis; every decaying mode stays in $\mathcal H$ at every order, the constraint being on growth alone.

\emph{Case $l = 1$} (any $k \geq 2$): the coordinate pair $(x_1,x_2)$ has leading order $r$ and is admissible, while $(x_3,x_4)$ has leading order $r^k$ and is not admissible; the boundary traces of $x_1,x_2$ are nonzero multiples of $\cos\theta$, $\sin\theta$ --- linearly independent --- so $X=\mathrm{span}\{e_1,e_2\}$ and $V = \mathrm{span}\{q(x_1), q(x_2)\}$ with $k = 2$. Theorem~\ref{thm:codim_discrete} gives discrete spectrum for any complementary Lagrangian subspace, and Theorem~\ref{thm:codim_coord_eigen} gives $\sigma = -1$ with multiplicity $\geq 2$ and eigenfunctions $x_1|_{\partial\Sigma}, x_2|_{\partial\Sigma}$; the non-admissible coordinates $x_3, x_4$ are \emph{not} claimed to be eigenfunctions. For $(k,l)=(2,1)$, the full Fraser--Sargent surface and its compact interior part have M\"obius-band topology, but the exterior piece $EFS_{2,1}$ is the image of the half-cylinder $[T_{2,1},\infty)\times S^1$ and is an annulus.

\emph{Case $l \geq 2$}: all four ambient coordinates grow at least like $r^2$, so $X = \{0\}$ at the order~\eqref{eq:adapted_window}; discreteness still holds, but Theorem~\ref{thm:codim_coord_eigen} applies only vacuously and no coordinate eigenfunction is asserted in this realization. Raising the order enlarges $\mathcal A$ and $X$; Theorem~\ref{thm:codim_coord_eigen} applies at every order, and all four coordinates are admissible once $N_1\geq k$. The supporting computations, including the orders needed for the $r^l$ coordinates and for all four coordinates, are collected in Appendix~\ref{app:FS_details}.

\emph{Relation to the Medvedev--Morozov observation.} Medvedev--Morozov \cite[§1.2, item~6]{MedvedevMorozov2022v2} note that \emph{``this Steklov problem does not have necessarily discrete Steklov spectrum since $\Sigma$ is not compact,''} which obstructs the definition of the spectral index. Theorem~\ref{thm:codim_discrete} removes the discreteness obstruction: at every choice of order, every complementary Lagrangian normalization of $EFS_{k,l}$ gives a self-adjoint realization with discrete spectrum and hence a finite counting function
\[
\mathcal N_{\mathcal L}(\tau):=\#\{\sigma_j(\mathcal S_{\mathcal L})<\tau\}
\]
for every fixed $\tau\in\mathbb R$. Choosing a geometrically meaningful threshold at which to call this count an exterior spectral index is a separate question.

The counting function depends on the normalization: by~\eqref{eq:counting} it has no negative part when $M\geq0$, in particular for the bounded realization $M=0$.

\appendix
\section{Deferred Details}\label{app:deferred}

\subsection{Huber--Osserman for surfaces with boundary}\label{app:huber}

Attach smooth compact Riemannian caps along $\partial\Sigma$.  Such caps exist for any smooth
boundary and the construction is intrinsic to $(\Sigma,g)$: in Fermi coordinates a collar metric
is exactly $du^2 + \varphi(u,\theta)^2\,d\theta^2$, and interpolating $\varphi$ to a constant $c$
and closing with the spherical cap $\varphi(u)=c\cos(u/c)$, smooth at the tip, asks nothing of the
boundary curve --- neither roundness nor constant geodesic curvature, and nothing about how it
sits in $S^2$.  The resulting complete surface has no boundary, agrees with $\Sigma$ outside a
compact set, and, the caps being compact, still has finite total negative curvature; Huber's
theorem \cite{Huber1957} therefore compactifies it by one puncture per end, and removing the caps
leaves the same punctured conformal neighborhoods of infinity.  The caps supply conformal type
only and need not be minimal: Osserman's meromorphic analysis of complete minimal surfaces of finite total
curvature~\cite{Osserman1964} is applied to the original minimal immersion on those
neighborhoods; its argument is local at each puncture, using only that a punctured neighborhood
is conformally a punctured disk on which the induced metric is complete and has finite total
curvature (finite total curvature excludes an essential singularity of the Gauss map, and
completeness then excludes one of the Weierstrass function $f$), and for the regular embedded multiplicity-one ends assumed here Schoen's theorem
\cite[Proposition~1]{Schoen1983} gives the planar/catenoidal dichotomy used in
Section~\ref{sec:multiple_ends}.  Neither argument rules out higher multiplicity;
Section~\ref{sec:higher_codim} treats that case under an explicit intrinsic hypothesis.

We record the outcome in the form used throughout.

\begin{lemma}[Conformal structure of a regular end]\label{lem:conformal_end}
Let $\Sigma\subset\mathbb R^3\setminus B^3$ be an exterior FBMS with compact boundary and finitely many regular ends. Then:
\begin{enumerate}[label=(\arabic*)]
\item each end $E_\ell$, shrunk if necessary, carries a conformal diffeomorphism $\zeta_\ell\colon E_\ell\to\{0<|\zeta|<\varepsilon\}$ with $(\zeta_\ell^{-1})^*g=\mu_\ell|d\zeta|^2$, $\mu_\ell>0$ smooth, and $\zeta_\ell\to0$ at infinity; we take $\zeta_\ell$ holomorphic for the orientation used in the Weierstrass representation below;
\item $\log r=-\log|\zeta_\ell|+O(1)$, the estimate being preserved by derivatives in the sense of Section~\ref{subsec:planar_ends};
\item after a rotation of the ambient frame taking the limit normal of $E_\ell$ to $\pm e_3$, $x_1-ix_2=c\,\zeta_\ell^{-1}+O(1)$ with $c\neq0$, and $x_3=a\log r+O(1)$ with $a$ the logarithmic growth rate of~\eqref{eq:catenoidal_polar}, $a=0$ for a planar end.
\end{enumerate}
In particular every ambient coordinate function lies in the space $\mathcal H$ of Definition~\ref{def:admissible} with orders $N_\ell=1$.
\end{lemma}

\begin{proof}
(1) is the Huber--Osserman compactification just described, one puncture per end.

For (2) and (3) recall the two structural facts we import: the Weierstrass data extend
meromorphically across the puncture \cite{Osserman1964}, and for an embedded end of
multiplicity one the horizontal Weierstrass differentials have double poles, so that the horizontal
coordinates have a simple-pole principal part \cite[Proposition~1 and its proof, p.~801]{Schoen1983}. Write
$\zeta=\zeta_\ell$ and use the convention
$\varphi_1=\tfrac12f(1-g^2)$, $\varphi_2=\tfrac i2f(1+g^2)$, $\varphi_3=fg$, so that
\begin{equation}\label{eq:phi_pm}
\varphi_1-i\varphi_2=f,\qquad \varphi_1+i\varphi_2=-fg^2 .
\end{equation}
Rotate the ambient frame so that the limit normal of $E_\ell$ is $\pm e_3$, the sign chosen so that $g(0)=0$, and
multiplicity one forces $f$ to have a pole of order exactly two, so
\[
g=O(\zeta),\qquad f=A\zeta^{-2}+B\zeta^{-1}+O(1),\qquad A\neq0 .
\]
We do \emph{not} assume the zero of $g$ to be simple: it is for a catenoidal end, it has order at
least two for a planar end not contained in a plane, and $g\equiv0$ for a plane. In every case
$fg^2=O(1)$, so by~\eqref{eq:phi_pm} the combination $\varphi_1+i\varphi_2$ is regular at the
puncture --- it equals $-1$ identically on the catenoid $f=\zeta^{-2}$, $g=\zeta$ --- while
$\varphi_1-i\varphi_2=f$ has a double pole.

The coefficient $B$ must vanish, and this is what makes the remainder in (3) bounded rather than
logarithmic. Indeed $\operatorname{Res}\varphi_1=\tfrac12B$ and
$\operatorname{Res}\varphi_2=\tfrac i2B$ by~\eqref{eq:phi_pm}, so going once around the puncture
$x_1$ gains $\operatorname{Re}(2\pi i\cdot\tfrac12B)=-\pi\operatorname{Im}B$ and $x_2$ gains
$\operatorname{Re}(2\pi i\cdot\tfrac i2 B)=-\pi\operatorname{Re}B$. Single-valuedness of the
immersion therefore gives $\operatorname{Im}B=\operatorname{Re}B=0$, that is $B=0$, and
\[
x_1-ix_2=\tfrac12\Bigl(\int f\,d\zeta+\overline{\int(-fg^2)\,d\zeta}\Bigr)
=-\frac{A}{2\zeta}+O(1),
\]
the second integral being bounded. This is (3) with $c=-A/2$. Finally $\varphi_3=fg$ has at worst
a simple pole, its residue is real because $x_3$ is single-valued, and the vertical part is
$a\log r+O(1)$ by~\eqref{eq:catenoidal_polar}; comparing $|\mathbf x|$ with $|\zeta|^{-1}$
gives (2).

The last sentence of the lemma follows since $-\log|\zeta|$, $\operatorname{Re}\zeta^{-1}$ and
$\operatorname{Im}\zeta^{-1}$ are exactly the principal parts~\eqref{eq:principal_parts} allowed
by $N_\ell=1$, and $\operatorname{Re}$ and $\operatorname{Im}$ of $c\,\zeta^{-1}$ span the same
real plane as those of $\overline{c\,\zeta^{-1}}$.
\end{proof}

\begin{remark}[Orientation]\label{rem:orientation}
In general neither of $x_1\pm ix_2$ is a holomorphic function (for the flat plane $x_1-ix_2=\zeta/2$
is): they are conjugate to one another, being combinations of real parts. What the rotation achieves is that the \emph{singular leading term}
of $x_1-ix_2$ is holomorphic, and that is the content of Lemma~\ref{lem:conformal_end}(3). On the
catenoid $f=\zeta^{-2}$, $g=\zeta$ one has exactly
\[
x_1-ix_2=-\frac{1}{2\zeta}-\frac{\bar\zeta}{2},\qquad
x_1+ix_2=-\frac{1}{2\bar\zeta}-\frac{\zeta}{2},
\]
so $\partial_{\bar\zeta}(x_1-ix_2)=-\tfrac12$: the pole sits in the first combination and the
bounded remainder is anti-holomorphic. On the plane $f\equiv1$, $g\equiv0$ one gets
$x_1+ix_2=\bar\zeta/2$, where $\zeta$ is the plane's own parameter and the puncture coordinate at
infinity is its reciprocal. Which of the two carries the pole depends on the orientation
convention and is immaterial below: only the real span
$\operatorname{span}_{\mathbb R}\{\operatorname{Re}\zeta_\ell^{-1},\operatorname{Im}\zeta_\ell^{-1}\}$
is used, and complex conjugation preserves it.
\end{remark}

\subsection{Kernel mechanics on the exterior of the unit disk}\label{app:kernel_example}

The maximum-principle argument in Lemma~\ref{lem:harmonic} excludes purely decaying Dirichlet-kernel elements: a harmonic function that vanishes on $\partial\Sigma$ and tends to zero on every end vanishes identically. If its growing coefficients vanish but a constant term remains, the function is bounded, extends across the punctures, and the same conclusion follows on the compactification. Thus at order $N_1=1$ the only possible leading kernel terms are $\log r$, $r\cos\phi$ and $r\sin\phi$.

On $C_0=\{r\geq1\}$, the exterior of the unit disk, the mechanism is explicit. The boundary condition at $r=1$ gives
\[
w_0=\log r,\qquad
w_1=(r-r^{-1})\cos\theta,\qquad
w_2=(r-r^{-1})\sin\theta.
\]
For $n\geq2$ the mode $r^{n}$ is not among the principal parts~\eqref{eq:principal_parts} admitted by $N_1=1$, and the remaining decaying coefficient is forced to zero by the boundary condition. Hence $K=\operatorname{span}\{w_0,w_1,w_2\}$. In particular, the notation $w_1=r\cos\phi+O(r^{-1})$ records only its leading asymptotic term, not a pure growing solution.

\subsection{Details for the Fraser--Sargent surfaces}\label{app:FS_details}

This subsection collects the computations supporting Section~\ref{subsec:FS_example}. The surfaces $EFS_{k,l}$ are given (on the end) by
\begin{align*}
(t, \theta) \mapsto \tfrac{1}{r_{k,l}}\bigl(&k\sinh(lt)\cos(l\theta),\; k\sinh(lt)\sin(l\theta),\\
&l\cosh(kt)\cos(k\theta),\; l\cosh(kt)\sin(k\theta)\bigr),
\end{align*}
with $k > l \geq 1$ and $t \in [t_0, \infty)$. Here $t_0 > 0$ is the unique positive root of the orthogonality condition, which in scale-invariant form reads $k\tanh(lt)\tanh(kt) = l$ and so does not involve the normalization, and the constant
\[
r_{k,l} := \bigl(k^2\sinh^2(lt_0) + l^2\cosh^2(kt_0)\bigr)^{1/2}
\]
is chosen so that $|\mathbf{x}| = 1$ on $\{t = t_0\}$; it encodes the half of the free-boundary condition that places $\partial\Sigma$ on $S^3$, the other half being the orthogonality.

\emph{Multiplicity of the end.} A direct computation shows that the coordinates $(t, \theta)$ are conformal: $\langle \mathbf{x}_t, \mathbf{x}_\theta \rangle = 0$ and $|\mathbf{x}_t|^2 = |\mathbf{x}_\theta|^2$, so the induced metric is
\begin{align*}
g &= \lambda(t)\,(dt^2 + d\theta^2),\\
\lambda(t) &= \frac{k^2 l^2}{r_{k,l}^2}\bigl(\cosh^2(lt) + \sinh^2(kt)\bigr) = c\,e^{2kt}\bigl(1 + O(e^{-2(k-l)t})\bigr).
\end{align*}
By Definition~\ref{def:intrinsic_end}, the end of $EFS_{k,l}$ is an intrinsic regular end of multiplicity $m = k$: intrinsically it is asymptotic to a cone of total angle $2\pi k$, not to a plane. Here $r = e^t$ is the \emph{intrinsic} conformal radial coordinate (the variable of Definition~\ref{def:intrinsic_end}); it is \emph{not} the ambient Euclidean radius $|\mathbf{x}|_{\mathbb{R}^4}$, which has order $r^k$ on the end. In this coordinate, the four ambient coordinate functions have the uniform expansions
\[
\begin{aligned}
(x_1,x_2)&=\frac{k}{2r_{k,l}}r^l\bigl(\cos(l\theta),\sin(l\theta)\bigr)+O(r^{-l}),\\
(x_3,x_4)&=\frac{l}{2r_{k,l}}r^k\bigl(\cos(k\theta),\sin(k\theta)\bigr)+O(r^{-k}).
\end{aligned}
\]

\emph{Admissibility.} In the conformal formulation the end is the punctured disk
$\zeta=e^{-(t+i\theta)}$, so $r=e^{t}=|\zeta|^{-1}$ and the growing harmonic modes are
$r^{n}\cos n\theta$, $r^{n}\sin n\theta$, $n\in\mathbb N$, together with $\log r$; these are the
real and imaginary parts of $\zeta^{-n}$, and by~\eqref{eq:multiplicity_admissibility} they are
retained precisely when $n\leq N_1$. The multiplicity $m=k$ does \emph{not} enter $\mathcal H$, the conformal factor being
invisible to it; it is visible only through the largest leading order $r^{k}$ of the
coordinates, hence as the smallest order $N_1$ at which all of them are admissible. At $N_1=1$ the admitted growing modes are exactly $1$, $\log r$ and
$r\cos\theta$, $r\sin\theta$, as in the $\mathbb{R}^3$ analysis: $r^{1}$ is admissible, $r^{2}$
is not. Decaying modes are unconstrained: $\operatorname{Re}\zeta^{n}$ and
$\operatorname{Im}\zeta^{n}$, that is $r^{-n}\cos n\theta$ and $r^{-n}\sin n\theta$, are bounded
at the puncture and lie in $\mathcal H$ at every order. What $N_1$ bounds is the growth retained,
and $q$ records only finitely many coefficients of the regular part.

\emph{Case $l = 1$ (any $k \geq 2$).} The displayed expansions show that $x_1,x_2$ have leading order $r$ and are admissible, whereas $x_3,x_4$ have leading order $r^k$ and are not, so
\[
X = \mathrm{span}\{e_1, e_2\}, \qquad V = \mathrm{span}\{q(x_1), q(x_2)\}.
\]
The boundary circle sits at the unique $t_0 > 0$ determined by the free-boundary condition ($\mathbf{x} \parallel \mathbf{x}_t$ on $\partial\Sigma$, i.e.\ $k\tanh(lt)\tanh(kt) = l$, which has exactly one positive root), so $x_1|_{\partial\Sigma}$ and $x_2|_{\partial\Sigma}$ are nonzero multiples of $\cos(l\theta)$ and $\sin(l\theta)$ --- linearly independent in $L^2(\partial\Sigma)$, so that Theorem~\ref{thm:codim_coord_eigen} applies with $k = 2$, as stated in Section~\ref{subsec:FS_example}. Although the full Fraser--Sargent surface is a M\"obius band when $k$ is even, the restriction $t\in[t_0,\infty)$ has no identification with the negative half-cylinder; hence every exterior piece $EFS_{k,l}$, including $EFS_{2,1}$, is an annulus.

\emph{Higher orders.} Including the first pair $x_1,x_2$, of leading order $r^{l}$, requires
$N_1\geq l$, and $N_1=l$ is the smallest such choice, retaining no \emph{growing} angular
harmonic beyond order $l$. For $l=1$ this is the order~\eqref{eq:adapted_window} itself, so nothing is gained there and
the discussion below concerns $l\geq2$. At $N_1=l$ the constant and logarithmic modes and all
growing angular harmonics of orders $1,\dots,l$ enter the asymptotic data, which then has
dimension $2(2l+1)$. Including also $x_3,x_4$, of leading order $r^{k}$, requires $N_1\geq k$;
the minimal choice $N_1=k$ admits all growing angular harmonics through order $k$ and gives
$\dim\mathcal A=2(2k+1)$. The construction of Section~\ref{sec:compactification} applies at any
order, but for $l\geq2$ the spaces $K$ and $\mathcal A$ and their Lagrangian dimensions are
larger than at $N_1=1$. Transversality and the existence of normalized extensions hold at every
order --- Proposition~\ref{prop:coordinate_eigenfunctions}(iv), Theorem~\ref{thm:krein} and Proposition~\ref{prop:boundary_normalization} --- so
no new theorem is required; only $X$, $k$, $N$ and the choice of normalization change. We do not
carry out either realization here.

\emph{The spectrum of one particular realization at $l=1$.} The orders alone do not determine
the spectrum: a realization is a choice of Lagrangian normalization, and different admissible
choices give different answers. We therefore fix one explicitly and compute only for it. Write
$\mu^{2}=\lambda(t_{0})$; the metric is conformally flat, so $\partial_\nu=-\mu^{-1}\partial_t$
and the harmonic modes are those of the flat cylinder. Define the normalized extension of each
boundary Fourier mode by
\begin{equation}\label{eq:efs_realization}
U_0(t)=1,\qquad
U_n(t)=\begin{cases}
\sinh t/\sinh t_0, & n=1,\\
\cosh(kt)/\cosh(kt_0), & n=k \text{ and } N_1\geq k,\\
e^{-n(t-t_0)}, & \text{otherwise,}
\end{cases}
\end{equation}
the same $U_n$ being used with $\cos n\theta$ and with $\sin n\theta$. This is a symmetric
realization --- the operator is diagonal in the real Fourier basis with real diagonal --- so by
Theorem~\ref{thm:krein} and Proposition~\ref{prop:boundary_normalization} it corresponds to a Lagrangian complement, and it is coordinate
compatible because the admissible coordinates are exactly the modes singled out in the first two
lines.

For this realization: the $n=0$ block gives $\partial_tU_0=0$ and $\sigma=0$; the $n=1$ block
gives $\sigma=-\coth(t_0)/\mu$; the $n=k$ block, when $N_1\geq k$, gives
$\sigma=-k\tanh(kt_0)/\mu$; every other block is decaying and gives $\sigma=n/\mu>0$. Both
$\coth t_0=\mu$ and $k\tanh(kt_0)=\mu$ reduce, on substituting $\mu^{2}=\lambda(t_0)$ and
$r_{k,l}^{2}$, to the same equation $k\tanh(t_0)\tanh(kt_0)=1$, which is the free-boundary
condition; so both coordinate eigenvalues equal $-1$ and the spectrum is, as multisets,
\[
\begin{cases}
\{-1,-1,0\}\cup\{n/\mu,\,n/\mu:n\geq2\}, & N_1<k,\\[2pt]
\{-1,-1,-1,-1,0\}\cup\{n/\mu,\,n/\mu:n\geq2,\ n\neq k\}, & N_1\geq k.
\end{cases}
\]
In particular the negative spectrum is exactly $\{-1\}$, of multiplicity equal to the number of
admissible coordinates, and $\sigma=0$ occurs. The zero is not a codimension phenomenon by
itself: the bounded realization $M=0$ has it on every surface considered here, and so does the
flat example $C_0$ of Example~\ref{ex:flat_disk}. The meaningful comparison is with the
coordinate normalization of a non-planar $C_\alpha$, $\alpha>0$, at $N_1=1$, where the constant
sector \emph{is} occupied --- by the axial coordinate, with eigenvalue $-1$ --- and no zero
occurs. On $EFS_{k,1}$ no coordinate has a constant trace (the pair $x_3,x_4$ sits in the $n=k$ sector), leaving the
constant sector free for~\eqref{eq:efs_realization} to fill with the constant extension.

\emph{Why the choice matters.} Keeping the coordinate blocks and replacing $U_0$ by
$U_0(t)=1+\kappa(t-t_0)$, which is harmonic, admissible for every $N_1\geq1$ and still has trace
$1$, gives $\sigma=-\kappa/\mu$ in the constant sector; the realization is again symmetric, so it
too is Lagrangian, and at $\kappa=\mu/2$ it has the extra negative eigenvalue $-\tfrac12$ while
every admissible coordinate still has eigenvalue $-1$. Likewise on $EFS_{3,1}$ at $N_1=3$ the
sector $n=2$ may be normalized by the growing mode $e^{2(t-t_0)}$, replacing $2/\mu$ by
$-2/\mu$. So the statement above is about~\eqref{eq:efs_realization} and not about
coordinate compatibility alone: what coordinate compatibility fixes is the eigenvalue $-1$ on the
admissible coordinates, not the rest of the spectrum.

\section*{Acknowledgements}

I am grateful to my advisor Vladimir Medvedev for his guidance, for many consultations, and for
his careful reading of and remarks on earlier versions of this paper. I am grateful to Iosif
Polterovich for his consultation and remarks on an earlier version: he pointed out that in
dimension two the exterior problem is best handled by conformal maps to punctured disks, supplied
a short draft in which the Kre\u\i n formula and its spectral consequences were first worked out,
and suggested the comparison with pseudo-Laplacians in Remark~\ref{rem:pseudo_laplacians}.

\section*{Tool and computational resource disclosure}

In accordance with the Leiden Declaration on Artificial Intelligence and
Mathematics, I record the computational tools used in preparing this paper.

\emph{Large language models.} Anthropic's Claude models and OpenAI Codex were used in three
distinct roles. (i)~\emph{Verification}: repeated multi-agent audits of the
manuscript, in which independent agents checked individual sections and
cross-cutting properties (hypothesis tracking through the text, notation
consistency, fidelity of cited statements), followed by adversarial passes in
which further agents attempted to refute each reported defect. (ii)~\emph{Drafting}:
several passages were drafted or redrafted with model assistance, among them the
derivation that a regular end has finite total curvature (Section~\ref{subsec:planar_ends}),
the capping construction of Appendix~\ref{app:huber}, and the conformal
compactification argument of Section~\ref{sec:compactification}. The restructuring of this
version around conformal maps to punctured disks follows a suggestion of I.~Polterovich, who
also supplied a short draft, prepared with model assistance, in which the Kre\u\i n formula and
its spectral consequences were worked out; Sections~\ref{sec:compactification}--\ref{sec:proof_discrete}
were rewritten from that suggestion and checked independently here. (iii)~\emph{Literature search}:
candidate references were located by semantic search over a theorem database and
then checked against the primary sources.

\emph{Responsibility and authorship.} All statements, proofs and citations have
been verified by the author, who retains sole responsibility for their
correctness. No automated system is credited as an author. Model output was
treated throughout as a draft requiring verification rather than as evidence:
the verification passes described above did identify errors in model-drafted
material, which were corrected before this version.

\emph{Attribution.} Where a model-assisted passage reproduces a standard
argument, I have attempted to identify and cite its source. Model output does
not carry provenance, so I cannot exclude that some routine steps coincide with
arguments in the literature that I have failed to attribute; I would be grateful
to have any such omission pointed out.

\end{document}